\documentclass[reqno]{amsart}

\usepackage{color}

\usepackage[english, activeacute]{babel}
\usepackage{amsmath,amsthm,amsxtra}
\usepackage{fix-cm}
\usepackage{graphics}
\usepackage{graphicx}	
\usepackage{fancyhdr}
\usepackage{float}
\usepackage{anysize}
\usepackage{thmtools}
\usepackage{amssymb,mathrsfs}
\usepackage{arydshln}
\usepackage{enumerate}
\numberwithin{equation}{section}

\usepackage{mathtools}
\usepackage{epsfig}
\usepackage{amssymb}
\usepackage{float}
\usepackage{latexsym}
\usepackage{amsfonts}
\usepackage{hyperref}

\usepackage{xfrac}
\usepackage[capitalize]{cleveref}

\newtheorem{theorem}{Theorem}[section]
\newtheorem*{theorem*}{Theorem}
\newtheorem{corollary}[theorem]{Corollary}

\newtheorem{lemma}[theorem]{Lemma}
\newtheorem{proposition}[theorem]{Proposition}

\theoremstyle{definition}
\newtheorem{definition}[theorem]{Definition}
\newtheorem{remark}[theorem]{Remark}

\newtheorem{example}[theorem]{Example}

\title[Smooth semi-Lipschitz functions and almost isometries of Finsler manifolds]{Smooth semi-Lipschitz functions and \\ almost isometries between Finsler manifolds}
\author{Aris Daniilidis}
\address{Departamento de Ingenier\'ia Matem\'atica, Universidad de Chile, CMM CNRS UMI 2807 (CONICYT AFB 170001), Av. Beauchef 851, Santiago, Chile}
\email{arisd@dim.uchile.cl}

\author{Jesus A. Jaramillo}
\address{Instituto de Matem\'atica Interdiscliplinar (IMI) and Departamento de An\'alisis Matem\'atico\\
	Universidad Complutense de Madrid\\
	28040-Madrid\\
	Spain}
\email{jaramil@mat.ucm.es}
\author{Francisco Venegas M.}
\address{Departamento de Ingenier\'ia Matem\'atica, Universidad de Chile, CMM CNRS UMI 2807 (CONICYT AFB 170001), Av. Beauchef 851, Santiago, Chile}
\email{fvenegas@dim.uchile.cl}

\thanks{The research of A. Daniilidis is supported by the research grants FONDECYT 1171854, CMM-CONICYT AFB 170001, ECOS-CONICYT C18E04 (Chile) and PGC2018-097960-B-C22 (Spain). The research of J. A. Jaramillo is supported in part by grant PGC2018-097286-B-I00 (Spain). The research of F. Venegas is supported by the CONICYT Doctorate Fellowship PFCHA/DOCTORADO NACIONAL/2019--21191167 and the FONDECYT Grant 1171854 (Chile).}

\begin{document}
\begin{abstract} The convex cone $SC_{\mathrm{SLip}}^1(\mathcal{X})$ of real-valued smooth semi-Lipschitz functions on a Finsler manifold $\mathcal{X}$ is an order-algebraic structure that captures both the differentiable and the quasi-metric feature of $\mathcal{X}$. In this work we show that the subset of smooth semi-Lipschitz functions of constant strictly less than $1$, denoted $SC_{1^{-}}^1(\mathcal{X})$, can be used to classify Finsler manifolds and to characterize \textit{almost isometries} between them, in the lines of the classical Banach-Stone and Mykers-Nakai theorems.
\end{abstract}
\maketitle

\section{Introduction}
Starting with the classical Banach-Stone Theorem, there is a long and fruitful line of research whose aim is to characterize the topological (respectively, metric, smooth) structure of a given space $X$ in terms of an algebraic or topological-algebraic structure on the space $C(X)$ of all real-valued continuous functions on $X$, or on a suitable subspace of $C(X)$.  We refer to the survey \cite{GJ} and references therein for further information about this subject. \smallskip
	
A variant of the preceding results is the so-called Myers-Nakai Theorem, stating that the Riemannian structure of a Riemannian manifold $\mathcal{X}$ is determined by the natural normed algebra $C^1_b(\mathcal{X})$ of all bounded $C^1$-smooth real functions on $\mathcal{X}$ with bounded derivative on $\mathcal{X}$. This result was initially proved by Myers~\cite{M} for compact manifolds and then extended by Nakai~\cite{N} to the general case. 
\smallskip

Let us recall that according to the Myers-Steenrod Theorem, the Riemannian structure of a Riemannian manifold is characterized in purely metric terms by its associated distance, in the sense that a bijection between two Riemann manifolds is a Riemann isometry if and only if it is a metric isometry for the corresponding Riemannian distances. The Myers-Steenrod Theorem has been extended to Finsler manifolds by Deng and Hou~\cite{DH}, and the results of the aforementioned work have in turn be used in~\cite{GJR-10} to extend the Myers-Nakai Theorem from the setting of Riemann manifolds to the one of {\it reversible} Finsler manifolds (see \cite[Theorem~3.1]{GJR-10} or forthcoming Theorem~\ref{4-thm3.1}). 
\smallskip

In this work we will focus on the case of general ({\it non-reversible}) Finsler manifolds (Definition~\ref{def-Finsler mnfd}). In this case, the associated distance is only a {\it quasi-metric}, in the sense that it does not need to be symmetric (see Definition~\ref{defqm}). In this setting, a natural class of transformations, considered in \cite{JLP}, are the so-called {\it almost isometries} (Definition~\ref{def-saliso}(i)) which are bijections between quasi-metric spaces that preserve the triangular functions. We recall that for a quasi-metric space $(X, d_{X})$ the associated {\it triangular function} is defined by $$Tr_X(x_1, x_2, x_3) := d_{X}(x_1, x_2) + d_{X}(x_2, x_3) - d_{X}(x_1, x_3),\quad\text{for all }x_1,x_2,x_3\in X,$$
and measures, in a sense, how far the involved points are from achieving equality in the triangle inequality. Every {\it isometry} (that is, a distance-preserving mapping) is an almost isometry and if the distances are symmetric (which is the case for reversible Finsler manifolds, or more generally, for metric spaces) the two classes coincide. The difference between isometries and almost isometries is illustrated as follows (see forthcoming Proposition~\ref{ai} for a more general formulation): a bijection  $\tau: \mathcal{X} \to \mathcal{Y}$ between Finsler manifolds is an almost isometry for the respective associated distances $d_{\mathcal{X}}$ and $d_{\mathcal{Y}}$ if, and only if, there exists a smooth function $\phi:\mathcal{X}\to\mathbb R$ (which is unique up to an additive constant) such that:
	\begin{equation}\label{qi}
	d_{\mathcal{Y}}(\tau (x_1), \tau (x_2)) = d_{\mathcal{X}} (x_1, x_2) + \phi (x_1) - \phi (x_2).
	\end{equation}
\smallskip
\emph{Strict almost isometries} form an intermediate class between isometries and almost isometries: an almost isometry $\tau:X\to Y$ between the quasi-metric spaces $(X, d_{X})$ and $(Y, d_{Y})$ is called {\it strict} if for some $c \geq 1$ it holds
$$
 {c}^{-1} \, d_{X}(x, x') \leq d_{Y}(\tau(x), \tau (x')) \leq c \, d_{X}(x, x'), \quad \text{  for all  } x, x'\in X.
$$

In the purely metric setting, a functional characterization of almost isometries between quasi-metric spaces has been obtained in \cite{CJ} using the {\it convex lattice structure} of the space of {\it (backward) semi-Lipschitz} functions (Definition~\ref{def-SL}(ii)) with semi-Lipschitz constant at most~$1$. In the smooth setting (whenever $\mathcal{X}$ has also a structure of a smooth manifold), we should naturally consider semi-Lipschitz functions that are additionally $\mathcal{C}^1$-smooth. However, this reveals an intrinsic difficulty, since no subclass of smooth functions can be given a lattice structure (differentiability is lost when taking suprema or infima). Moreover, it is more natural to consider {\it forward} (rather than backward) {\it semi-Lipschitz} functions ({\it c.f.} Definition~\ref{def-SL}(i)), since for these functions the semi-Lipschitz constant coincides with supremum of the asymmetric norms of their derivatives (Corollary~\ref{Slip-dif}). Last, but not least, for reasons that we shall figure out later, the natural morphisms between general (non-reversible, non-compact) Finsler manifolds are the strict almost isometries, rather than almost isometries. (In Proposition~\ref{compact} we shall see that every almost isometry between compact Finsler manifolds is in fact strict.) 
\smallskip

Resuming the above, given a Finsler manifold $\mathcal{X}$, we shall work with the structure $SC^1_{1^-}(\mathcal{X}, d_{\mathcal{X}})$ of all $C^1$-smooth, (forward) semi-Lipschitz functions on $\mathcal{X}$ with semi-Lipschitz constant strictly less than~$1$, considered as a {\it partially ordered convex set}. Our main result, in this work, is to establish a functional characterization of strict almost isometries between Finsler manifolds, in terms of {\it convex-order isomorphisms} between the class of smooth {\it (forward) semi-Lipschitz} functions of constant strictly less than~1 (or, equivalently, smooth functions for which the supremum of the asymmetric norm of their derivatives is strictly less than~1.) 
\smallskip

Indeed, given a strict almost isometry $\tau:\mathcal{X} \to \mathcal{Y}$ between Finsler manifolds, it turns out (see  Proposition~\ref{isomorph}) that the function $\phi:\mathcal{X}\to\mathbb R$ associated to $\tau$ in the sense of equation \eqref{qi} determines a bijection $T:SC^1_{1^-}(\mathcal{Y}, d_{\mathcal{Y}})\rightarrow SC^1_{1-}(\mathcal{X},d_{\mathcal{X}})$ in the following way: $f\mapsto Tf= f\circ\tau+\phi$. It is clear that this map preserves both order and convex combinations. In the sequel we refer to this mapping as an {\em isomorphism of convex partially ordered sets}.

\smallskip

In the opposite direction, let us note that there are three main types of natural isomorphism between convex partially ordered sets, namely: \smallskip
	\begin{itemize}
		\item $T_1: SC^1_{1^-}(\mathcal{Y}, d_{\mathcal{Y}})\rightarrow SC^1_{1^-}(\mathcal{X},d_{\mathcal{X}}), T_1f=f\circ\tau$,
		where $\tau$ is an isometry. \smallskip
		\item $T_2:SC^1_{1^-}(\mathcal{X},d_{\mathcal{X}})\rightarrow SC^1_{1^-}(\mathcal{X}, c \, d_{\mathcal{X}}), T_2f=c \cdot f$,
		where $c \in(0,\infty)$. \smallskip
		\item $T_3:SC^1_{1^-}(\mathcal{X},d'_{\mathcal{X}})\rightarrow SC^1_{1^-}(\mathcal{X},d_{\mathcal{X}}),
		T_3f=f+\phi$, where $d'(x,x')=d(x,x')+\phi(x)-\phi(x')$.
	\end{itemize}
\smallskip
	In Theorem~\ref{teo} we show that given two connected, second countable and bicomplete Finsler manifolds $\mathcal{X}$ and $\mathcal{Y}$, every isomorphism of convex partially ordered sets $T:SC^1_{1^-}(\mathcal{Y},d_{\mathcal{Y}})\rightarrow SC^1_{1^-}(\mathcal{X},d_{\mathcal{X}})$  is, in fact, a composition of one of each kind: $T$ is of the form $Tf=c \cdot (f\circ\tau) +\phi$.
\medskip

	\subsection{Organization of the paper}
	In Section~\ref{prelim} we recall definitions and previous results, regarding quasi-metric spaces, Finsler manifolds and semi-Lipschitz functions. In Section~\ref{main} we present the proof of our main result as well as several consequences.

	\section{Preliminaries}\label{prelim}
	In this article, we denote by $\mathbb{R}$ the set of real numbers. For any two numbers $s,t\in \mathbb{R}$, we denote by $s\vee t$ (respectively, $s\wedge t$) the maximum (respectively, the minimum) of $s$ and $t$.
	\subsection{Quasi-metric spaces}\hfill

We start by recalling the definition of a quasi-metric space. (The reader should be advertised that this terminology is not universal: some authors consider variants of this definition allowing the quasi-metric to take negative values and/or the value $+\infty$.)
	\begin{definition}[Quasi-metric space]\label{defqm}
		A \emph{quasi-metric space} is a pair $(X,d)$, where $X$ is a nonempty set and ${d:X\times X\to [0,\infty)}$ is a function satisfying:
		\begin{enumerate}
			\renewcommand\labelenumi{(\roman{enumi})}
			
			\item $d(x,x)=0$ for all $x\in X$. \smallskip
			\item $d(x,y)=d(y,x)=0$ implies $x=y$ for all $x,y\in X$. \smallskip
			\item $d(x,y)\leq d(x,z)+d(z,y)$ for any $x,y,z\in X$.\smallskip
		\end{enumerate}
		Condition (iii) corresponds to the \emph{triangular inequality}. Replacing (ii) by the stronger condition \smallskip
		\begin{enumerate}
			\item [(ii)']  $d(x,y)=0\implies x=y$, \smallskip
		\end{enumerate}
		we get the definition of a \emph{$T_1$-quasi-metric space}.
	\end{definition}

A quasi-metric space need not be $T_2$ (neither $T_1$), but its symmetrization satisfies both properties, since it yields a metric space.
	\begin{definition}[Symmetrized distance]\label{deftopo}
		For a quasi-metric space $(X,d)$, the \emph{reverse quasi-metric} $\bar{d}$ is defined by $\bar{d}(x,y)=d(y,x)$, and the \emph{symmetrized distance} $d^s$ is defined by ${d^s(x,y)=d(x,y)\vee d(y,x)}$. Clearly, $\bar{d}$ is a quasi-metric and $d^s$ is a metric.
   \end{definition}
	The following definition describes the topologies that are naturally associated to a quasi-metric space.
\begin{definition}[Topologies of a quasi-metric space]\label{def-top}
To each quasi-metric space $(X,d)$ we can associate three ``natural'' topologies:
	\begin{enumerate}
		\renewcommand\labelenumi{(\roman{enumi})}
		\item the \emph{forward topology} $\mathcal{T}(d)$, generated by the
		family of open \emph{forward}-balls \\ $\{B_{d}(x,r)\hbox{\rm :}\ x\in X,$
		$r>0\},$ with ${B_{d}(x,r)=\{y\in X\hbox{\rm :}\ d(x,y)<r\}}$ for
		any $x\in X$ and $r>0.$ \medskip
		\item the \emph{backward topology} $\mathcal{T}(\bar{d})$, generated by the family of
        \emph{backward}-balls:\\ ${B_{\bar{d}}(x,r)=\{y\in X\hbox{\rm :}\ d(y,x)<r\}}$ for any $x\in X$ and $r>0.$ \medskip
		\item the \emph{symmetric topology} $\mathcal{T}(d^s)$, which is the metric topology induced by the distance $d^s$,\\ or equivalently, by the family ${\{B_{d}(x,r)\cap B_{\bar{d}}(x,r)\hbox{\rm :}\ x\in X,r>0\}}$.
   \end{enumerate}
\end{definition}

	\smallskip
There are several ways to consider a notion of completeness for a quasi-metric space, the most forward one being the usual (metric) completeness of the symmetrized (metric) space. Following terminology of the recent literature, we refer to this notion as \emph{bicompleteness} of the quasi-metric space.

	\begin{definition}[Bicompleteness]\cite{C}
		A quasi-metric space $(X,d)$ is said to be \emph{bicomplete} if the metric space $(X,d^s)$ is complete.
	\end{definition}
	\medskip
We shall now define the notion of an \emph{almost isometry} (a weaker notion than mere \emph{isometry}) to identify structure of quasi-metric spaces, which is based on the notion of triangular function.
	\begin{definition}[Triangular function]\label{defai}
Let $(X,d)$ be a quasi-metric space. The \emph{triangular function} $\mathrm{Tr}_X:X\times X \times X\to [0,+\infty)$ (associated to the quasi-metric space $X$) is defined by $${\mathrm{Tr}_X(x_1,x_2,x_3)=d(x_1,x_2)+d(x_2,x_3)-d(x_1,x_3)}.$$
\end{definition}

\begin{definition}[(strict) almost isometries]\label{def-saliso}		
A bijection $\tau:X\to Y$ between the quasi-metric spaces $(X,d_X)$ and $(Y,d_Y)$ is called: 
\begin{itemize}
\item[(i)] an \emph{almost isometry}, if it preserves the respective triangular functions, that is 
\begin{equation}\label{eq:alm-iso} 
\mathrm{Tr}_Y(\tau(x_1),\tau(x_2),\tau(x_3))=\mathrm{Tr}_X(x_1,x_2,x_3), \quad\text{for all }x_1,x_2,x_3\in X
\end{equation}
\item[(ii)]
a \emph{strict almost isometry}, if it satisfies \eqref{eq:alm-iso} and there exists a constant $c \geq 1$ such that
$$
 \frac{1}{c} \, d_{X}(x_1, x_2) \leq d_{Y}(\tau(x_1), \tau (x_2)) \leq c \, d_{X}(x_1, x_2) \quad \text{  for all  } x_1, x_2\in X.
$$
\item[(iii)] an \emph{isometry}, if $d_X(x_1,x_2)=d_Y(\tau(x_1),\tau(x_2))$ for any $x_1,x_2\in X$.
\end{itemize}
	\end{definition}
	Clearly, every isometry is a (strict) almost isometry, and in metric spaces every almost isometry is in fact an isometry and the three notions above coincide. The following characterization of almost-isometries was obtained in \cite[Proposition~2.8]{JLP}.
	\medskip
	\begin{proposition}[Characterization of almost isometries]\label{ai}
Given quasi-metric spaces $(X,d_X)$ and $(Y,d_Y)$, a bijection $\tau:X\to Y$ is an almost isometry if and only if there exists a function $\phi:X\to \mathbb{R}$ such that for any $x_1,x_2\in X$
		$$d_Y(\tau(x_1),\tau(x_2))=d_X(x_1,x_2)+\phi(x_1)-\phi(x_2).$$
		Moreover, the function $\phi$ can be determined up to an additive constant by $$\phi(x)=d_Y(\tau(x),\tau(x_0)) - d_X(x,x_0), \quad \text{for any fixed } x_0\in X. $$
	\end{proposition}
	\medskip
The forthcoming notion of \emph{(forward/backward) semi-Lipschitz} function consists of the class of natural real-valued morphisms defined on a quasi-metric space, that capture its structure. We recall the definition below.

\begin{definition}[Semi-Lipschitz function]\label{def-SL}
Let $(X,d)$ be a quasi-metric space. \smallskip \\
(i). A function $f:X\to \mathbb{R}$ is called \emph{forward semi-Lipschitz} (or simply \emph{semi-Lipschitz}) if there exists $L\geq 0$ such that 	 
\begin{equation}\label{eq:fSL}
f(y)-f(x)\leq L\,d(x,y),\quad\text{for all } x,y\in X. 
\end{equation}
The infimum of the above constants $L>0$ is called the \emph{(forward) semi-Lipschitz constant} of $f$, that is,
		$$\|f|_S:=\sup_{d(x,y)>0}\frac{f(y)-f(x)}{d(x,y)}.$$
We denote by $\mathrm{SLip}(X,d)$ (or simply, $\mathrm{SLip}(X)$) the set of (forward) semi-Lipschitz functions on $(X,d)$.
\smallskip \newline

\noindent (ii). A function $f:X\to \mathbb{R}$ is said to be \emph{backward semi-Lipschitz} if $-f$ is (forward) semi-Lipschitz, or equivalently, if there exists $L\geq 0$ such that $f(x)-f(y)\leq L\,d(x,y)$, for all $x,y\in X.$ The infimum of the above constants $L>0$ is called the \emph{(backward) semi-Lipschitz constant} of $f$ that is,
		$$|f\|_S:=\sup_{d(x,y)>0}\frac{f(x)-f(y)}{d(x,y)}.$$
Notice that $f$ is backward semi-Lipschitz on $(X,d)$ if and only if $f$ is (forward) semi-Lipschitz on $(X,\bar d)$ (the reverse quasi-metric). Therefore, we shall denote by $\mathrm{SLip}(X,\bar d)$ (or simply, $\mathrm{SLip(\bar{X}}$) the set of backward semi-Lipschitz functions on $(X,d)$.
\smallskip \newline

\noindent (iii). A function $f:X\to \mathbb{R}$ is Lipschitz, if there exists $L\geq 0$ such that for any $x,y\in X$ $$|f(x)-f(y)|\leq Ld(x,y).$$
The \emph{Lipschitz constant} of $f$ is defined by
$$ \|f\|_{\mathrm{Lip}}:=\sup_{d(x,y)>0}\frac{|f(x)-f(y)|}{d(x,y)}=\max\{\|f|_S,\,|f\|_S\}.$$
	\end{definition}

\begin{remark} (i). We use the notation $\|\cdot|_S$ (with double bar on the left and only one bar on the right) and respectively, $|\cdot\|_S$, to indicate that this is an asymmetric quantity, in the sense that even if both $f$ and $-f$ are semi-Lipschitz, we typically have $\|f|_S\neq \|-f|_S$. Notice that if $(X,d)$ is a metric space, semi-Lipschitz and Lipschitz functions are the same, and $\|f|_S = |f\|_S=\|f\|_{\mathrm{Lip}}$. In a quasi-metric space, a function $f$ is Lipschitz if and only if both $f$ and $-f$ are semi-Lipschitz. In this case, the Lipschitz constant of $f$ is the maximum of the semi-Lipschitz constants of $f$ an $-f$.  \smallskip\\
(ii). Canonical examples of (forward) semi-Lipschitz functions that are not Lipschitz are functions of the form $d(x_0, \cdot)$ on certain type of quasi-metric space (see \cite[Example~2.3]{CJ} \textit{e.g}).\smallskip \\
(iii). A semi-Lipschitz function on a quasi-metric space $(X,d)$ might not be ``forward continuous" (that is, continuous with respect to the forward topology, see Definitiion~\ref{def-top}(i)). However, it is Lipschitz (and therefore continuous) on the associated symmetrized metric space $(X,d^s)$. 
	\end{remark}

We now give a simple characterization of strict almost isometries, which will be useful in the sequel.

\begin{proposition}[Characterization of strict almost isometries]\label{sai}
Let $\tau:X\to Y$ be an almost isometry between the quasi-metric spaces $(X,d_X)$ and $(Y,d_Y)$. Let  $\phi:X\to \mathbb{R}$ and $\psi:Y\to \mathbb{R}$ be the functions associated to $\tau$ and respectively, to $\tau^{-1}$ in the sense of Proposition~\ref{ai}, that is,
		$$ d_Y(\tau(x_1),\tau(x_2))=d_X(x_1,x_2)+\phi(x_1)-\phi(x_2), $$
        $$ d_X(\tau^{-1}(y_1),\tau^{-1}(y_2))=d_Y(y_1,y_2)+\psi(y_1)-\psi(y_2).$$
Then $\tau$ is a strict almost isometry if, and only if, $\|\phi|_S <1$ and $\|\psi|_S <1$.
\end{proposition}
\begin{proof}
Suppose first that $\tau:X\to Y$ is a strict almost isometry, and consider $c > 1$ such that
$$
 c^{-1} \, d_{X}(x, x') \leq d_{Y}(\tau(x), \tau (x')) \leq c \, d_{X}(x, x') \quad \text{  for all  } x, x'\in X.
$$
Since $\phi(x')-\phi(x) =  d_X(x,x') - d_Y(\tau(x),\tau(x')),$ whenever $d_{X}(x, x')>0$ we have that
$$
\frac{\phi(x')-\phi(x)}{d_X(x,x')} = 1 - \frac{d_Y(\tau(x),\tau(x'))}{d_X(x,x')} \leq 1 - c^{-1}.
$$
Thus $\|\phi|_S \leq 1 - c^{-1} <1$. By considering $\tau^{-1}$, we also obtain that $\|\psi|_S \leq 1 - c^{-1} <1$.
\smallskip

Conversely, let $0< \alpha <1$ such that $\|\phi|_S \leq \alpha$ and $\|\varphi|_S  \leq \alpha$. Then for $d_{X}(x, x')>0$ we have that
$$
\frac{d_Y(\tau(x),\tau(x'))}{d_X(x,x')} = 1 - \frac{\phi(x')-\phi(x)}{d_X(x,x')} \geq 1 - \alpha = \frac{1}{c},
$$
where $c= (1-\alpha)^{-1}$. The other inequality follows in the same way.
\end{proof}

Semi-Lipschitz functions (respectively, backward semi-Lipschitz functions) are stable with respect to the max/min operations. We have in particular the following definition.

\begin{definition}[The convex lattice $\mathrm{SLip}_1(\bar{X})$]
Let $(X,d)$ be a quasi-metric space. The space of backward semi-Lipschitz functions with backward semi-Lipschitz constant less or equal to 1 is denoted by
$$\mathrm{SLip}_1(\bar X)=\{f:X\to \mathbb{R}\::\:f(x)-f(y)\leq d(x,y)\}\,\left(=\{f:X\to \mathbb{R}\::\:f(y)-f(x)\leq \bar d(x,y)\}\,\right).$$
	\end{definition}
	\noindent It is not difficult to check that given $f,g\in \mathrm{SLip}_1(\bar X)$, both their supremum $f\vee g$ and their infimum $f\wedge g$ belong to $\mathrm{SLip}_1(\bar X)$, so $\mathrm{SLip}_1(\bar X)$ has a natural lattice structure. Furthermore, it is also closed under convex combinations. Thus, following \cite{CJ}, we say that $\mathrm{SLip}_1(\bar X)$ has a \emph{convex lattice} structure. If $(Y,\rho)$ is another quasi-metric space, we say that a bijection $T:\mathrm{SLip}_1(\bar Y)\to \mathrm{SLip}_1(\bar X)$ is a \emph{convex lattice isomorphism} if $T$ preserves both order and convex combinations, that is, \smallskip
	\begin{itemize}
		\item $Tf\geq Tg$ if and only if $f\geq g$ for all $f,g\in \mathrm{SLip}_1(\bar Y)$, and \smallskip
		\item $T(\lambda f+(1-\lambda)g)=\lambda Tf +(1-\lambda)Tg$ for all $f,g\in \mathrm{SLip}_1(\bar Y)$ and $\lambda \in [0,1]$.
	\end{itemize}

\begin{remark}
Note that any order-preserving bijection between lattices is automatically a lattice isomorphism, so any convex lattice isomorphism satisfies $T(f\wedge g)=Tf\wedge Tg$ and $T(f\vee g)=Tf\vee Tg$ for all $f,g\in \mathrm{SLip}_1(\bar Y)$.
\end{remark}
	\medskip

	The following result, taken from \cite[Theorem 3.1]{CJ}, reveals the importance of the convex lattice structure $\mathrm{SLip}_1(\bar X)$ for the study of the quasi-metric structure of a bicomplete quasi-metric space.

	\begin{theorem}[representation of almost isometries between quasi-metric spaces]\label{4-thm3.1}
		Let $(X,d)$ and $(Y,\rho)$ be bicomplete quasi-metric spaces, and let  ${T:\mathrm{SLip}_1(\bar Y)\to \mathrm{SLip}_1(\bar X)}$ be a convex lattice isomorphism. Then there exist $\alpha>0$, an homeomorphism ${\tau:(\bar X,d)\to (\bar Y,\rho)}$ and a quasi-metric $d'$ on $X$, such that \smallskip
		\begin{itemize}
			\item $(X,d)$ and $(X,d')$ are almost-isometric, and $d'(x,x')=d(x,x')+T0(x')-T0(x).$ \smallskip
			\item $\tau:(X,\alpha \cdot d')\to (Y,\rho)$ is an isometry. \smallskip
			\item For every $f\in \mathrm{SLip}_1(\bar Y)$ we have that $Tf=c\cdot(f\circ \tau)+\phi$, where $c={\alpha}^{-1}$ and $\phi=T0$.
		\end{itemize}
	\end{theorem}
	Therefore, two bicomplete quasi-metric spaces are almost isometric up to a multiplicative constant whenever the respective spaces of $1$-backward semi-Lipschitz functions are \emph{isomorphic as convex lattices}, and that the isomorphism is a composition operator associated with the almost isomerty. Notice that (forward) semi-Lipschitz functions with semi-Lipschitz constant less or equal to 1 form readily an analogous convex lattice structure. In particular, the above theorem can be readily restated in a completely analogous way in terms of the lattices $\mathrm{SLip}_1(X)$ and $\mathrm{SLip}_1(Y)$. \smallskip

In this work we establish a result of similar flavor to the above, in case that the quasi-metric spaces are \emph{Finsler manifolds} (see forthcoming Definition~\ref{def-Finsler mnfd}). The structure that is naturally associated to this study are (forward) smooth semi-Lipschitz functions. As already mentioned in the introduction, the main difficult in this framework is that the operations $f\wedge g$ and $f\vee g$ are not compatible with differentiability and, as a consequence, we do no longer have a lattice structure.

	\medskip
	\subsection{Finsler manifolds}
	\begin{definition}[Minkowski norm]
		Let $V$ be a finite-dimensional real vector space. A functional $F:V\to [0,+\infty)$ is called a \emph{Minkowski norm} on $V$ if the following conditions are satisfied:\smallskip
		\begin{enumerate}
			\renewcommand\labelenumi{(\roman{enumi})}
			\item Positive homogeneity: $F(\lambda v)=\lambda F(v)$ for every $v\in V$ and $\lambda \geq 0$.\smallskip
			\item Regularity: $F$ is continuous on $V$ and $C^\infty$-smooth on $V\setminus \{0\}$.\smallskip
			\item Strong convexity: for every $v\in V\setminus \{0\}$, the quadratic form associated to the second derivative of the function $F^2$ at $v$, that is, $$g_v=\frac{1}{2}d^2[F^2](v),$$
			is positive definite on $V$.
		\end{enumerate}
Every Minkoweki norm satisfies in addition the following conditions (see \cite[Theorem~1.2.2]{BCS} \textit{e.g.}):\smallskip
\begin{enumerate}
 			\item[(iv)] Positivity: $F(v)=0$ if and only if $v=0$. \smallskip
			\item[(v)] Triangle inequality: $F(u+v)\leq F(u)+F(v)$, for every $u,v\in V$.
\end{enumerate}
 It is clear that every norm associated to an inner product is a Minkowski norm.
		In general, a Minkowski norm does not need to be symmetric, and there are indeed very interesting examples of asymmetric Minkowski norms, such as, for example, Randers spaces (\cite{BCS}) or more generally Finsler manifolds.   \smallskip\newline
We say $F$ is \emph{symmetric} (or \emph{absolutely homogeneous}) if
		$$F(\lambda v)=|\lambda|F(v)\text{ for any }\lambda\in \mathbb{R} \text{ and }v \in V.$$
		In this case, $F$ is a norm in the usual sense.
	\end{definition}
	\medskip
	\begin{definition}[Finsler manifold]\label{def-Finsler mnfd}
		A Finsler manifold is a pair $(\mathcal{X},F)$ such that $\mathcal{X}$ is a finite-dimensional $C^\infty$-smooth manifold and $F:T\mathcal{X}\to [0,\infty)$ is a continuous function defined on the tangent bundle $T\mathcal{X}$, satisfying \smallskip
		\begin{enumerate}
			\renewcommand\labelenumi{(\roman{enumi})}
			
			\item $F$ is a $C^\infty$-smooth on $T\mathcal{X}\setminus\{0\}.$ \smallskip
			\item For every $x\in \mathcal{X}$, $F(x,\cdot):T_x\mathcal{X}\to [0,\infty)$ is a Minkowski norm on the tangent space $T_x\mathcal{X}$.
		\end{enumerate}
	\end{definition}

		The Finsler structure $F$ is said to be \emph{reversible} if, for every $x\in \mathcal{X}$, $F(x,\cdot)$ is symmetric. Clearly, any Riemannian manifold is a reversible Finsler manifold, where the symmetric Minkowski norm on each tangent space is given by an inner product.
	\begin{definition}[Finsler distance $d_F$]
		Let $(\mathcal{X},F)$ be a connected Finsler manifold. The \emph{Finsler distance} $d_F$ on $\mathcal{X}$ is defined by
		$$d_F(x,y)=\inf\{\ell_F(\sigma)\,:\,\sigma \text{ is a piecewise }C^1\text{ path from }x\text{ to }y\},$$
		where the Finsler length of a piecewise $C^1$ path $\sigma:[a,b]\to \mathcal{X}$ is defined as:
		$$\ell_F(\sigma)=\int_a^bF(\sigma(t),\dot{\sigma}(t))dt,$$
		\noindent where $\dot{\sigma}$ is the derivative of $\sigma$.
		The Finsler distance $d_F$ is a $T_1$-quasi-metric on $\mathcal{X}$ for any connected Finsler manifold $(\mathcal{X},F)$ (see \textit{e.g.} \cite[Section~6.2]{BCS}).
	\end{definition}
\medskip
\begin{remark}[Topology of a Finsler manifold]\label{topfinsler}
	Even if the forward and backward distances of a connected Finsler manifold $\mathcal{X}$ differ, they do induce the same topology on $\mathcal{X}$, which coincides with the manifold topology (see \cite[Chapter~6.2]{BCS}). Therefore, for Finsler manifolds, the three topologies of Definition~\ref{deftopo} are the same.
\end{remark}

\medskip
\begin{definition}[Finsler isometry]
	A mapping $\tau:(\mathcal{X},F)\to(\mathcal{Y},G)$ between Finsler manifolds is said to be a \emph{Finsler isometry} if it is a diffeomorphism which preserves the Finsler structure, that is, for every $x\in \mathcal{X}$ and every $v\in T_x\mathcal{X}$:
	$$F(x,v)=G(\tau(x),d\tau(x)(v)).$$
\end{definition}

A classical result due to Myers and Steenrod \cite{MS} asserts that a mapping between Riemannian manifolds is a Riemannian isometry if and only if it is a metric isometry for the corresponding Riemannian distances. This was extended by Deng and Hou in \cite{DH} to the context of Finsler manifolds:
\medskip
\begin{theorem}[Characterization of isometries for Finsler manifolds]
	Let $(\mathcal{X},F)$ and $(\mathcal{Y},G)$ be connected Finsler manifolds. Then $\tau:(\mathcal{X},F)\to(\mathcal{Y},G)$ is a Finsler isometry if and only if it is bijective and an isometry for the corresponding Finsler distances.
\end{theorem}

A weaker result, established in \cite{JLP} (see Lemma~3.1 and Proposition~3.2 therein), holds for almost isometries. Given a diffeomorphism $\tau:\mathcal{X}\to \mathcal{Y}$ and a Finsler structure $F$ on $\mathcal{X}$, we denote by $\tau_*(F)$ the Finsler structure on $\mathcal{Y}$ obtained as the push-forward of $F$ by $\tau$, that is, for every $y\in \mathcal{Y}$ and every $w\in T_y\mathcal{Y}$:
$$
\tau_*(F)(y,w)=F(\tau^{-1}(y), d\tau^{-1}(y)(w)).
$$

\begin{proposition}[Characterization of almost isometries for Finsler manifolds]\label{tausmooth}
	Let $(\mathcal{X},F)$ and $(\mathcal{Y},G)$ be connected Finsler manifolds, and let $\tau:\mathcal{X}\to \mathcal{Y}$ be an almost isometry induced by a function $\phi:\mathcal{X}\to \mathbb{R}$ (in the sense of Proposition~\ref{ai}). Then $\tau$ and $\phi$ are smooth, and $G=\tau_*(F) - d(\phi\circ \tau^{-1})$. Conversely, if $G=\tau_*(F) - d(\phi\circ \tau^{-1})$, then $\tau$ is an almost isometry.
\end{proposition}

\medskip
In what follows, for simplicity, the term Finsler manifold will also refer to the pair $(\mathcal{X},d_{\mathcal{X}})$, where $(\mathcal{X},F)$ is a Finsler Manifold and $d_{\mathcal{X}}$ is the Finsler distance induced by $F$.
\medskip


\subsection{Smooth semi-Lipschitz functions}
We shall now introduce a class of real-valued functions that is naturally associated to Finsler manifolds.

\begin{definition}[The convex partially ordered set $SC^1_{1^-}(\mathcal{X})$]
Let $(\mathcal{X},d_{\mathcal{X}})$ be a connected Finsler manifold. The space of $C^1$-smooth (forward) semi-Lipschitz functions with semi-Lipschitz constant \emph{strictly} less than $1$ will be denoted by
$$SC_{1^{-}}^1(\mathcal{X}):=\{f\in C^1(X,\mathbb{R})\::\: \|f|_S<1\}.$$
\end{definition}

When the Finsler manifold $(\mathcal{X},d_{\mathcal{X}})$ is reversible, we write $C_{1^-}^1(\mathcal{X})$ instead of $SC_{1^-}^1(\mathcal{X})$.\smallskip

The set $SC_{1^-}^1(\mathcal{X})$ (respectively, the set $C_{1^-}^{1}(\mathcal{X})$ in the reversible case) is convex and partially ordered, but in contrast to $\mathrm{SLip}_1(\mathcal{X})$, it is not a lattice, since differentiability is lost when taking suprema and infima. Therefore, for the study of Finsler manifolds, we shall consider the structure $SC_{1-}^1(\mathcal{X})$ as a \emph{convex partially ordered} set. We shall now define the notion of isomorphism for the aforementioned structures.

\begin{definition}[Isomorphism between convex partially ordered sets]
Given connected Finsler manifolds $(\mathcal{X},d_{\mathcal{X}})$ and $(\mathcal{Y},d_{\mathcal{Y}})$, we say that a bijection $$T:SC_{1^-}^1(\mathcal{Y})\to SC_{1^-}^1(\mathcal{X})$$ is an \emph{isomorphism of convex partially ordered sets} if \smallskip
\begin{itemize}
    \item[(i)] $Tf\geq Tg$ if and only if $f\geq g$ for all $f,g\in SC_{1^-}^1(\mathcal{Y})$, and \smallskip
    \item[(ii)] $T(\lambda f+(1-\lambda)g)=\lambda Tf +(1-\lambda)Tg$ for all $f,g\in SC_{1^-}^1(\mathcal{Y})$ and $\lambda \in [0,1]$.
\end{itemize}
\end{definition}

We shall now define the norm and the asymmetric norm of the derivative $df(x)$ of a smooth function $f\in C^1(\mathcal{X})$, at a point $x$ of  a Finsler manifold $\mathcal{X}$.

\begin{definition}[Norm and asymmetric norm of the derivative $df(x)$]
Let $(\mathcal{X},F)$ be a connected Finsler manifold and $f:\mathcal{X}\to \mathbb{R}$ a $C^1$-smooth function. The norm of the derivative of $f$ at the point $x\in \mathcal{X}$ is defined by:
$$\|df(x)\|_F=\sup\{|df(x)(v)|\::\:v\in T_x\mathcal{X},\:F(x,v)\leq1\}.$$

\smallskip

In the same way, the asymmetric norm of $df(x)$ is defined by:
$$\|df(x)|_F=\sup\{df(x)(v)\::\:v\in T_x\mathcal{X},\:F(x,v)\leq1\}.$$
\end{definition}

It is clear that, in the case of a reversible Finsler manifold, the norm and the asymmetric norm of $df(x)$ coincide. In general, we have that $\|df(x)|_F \leq \|df(x)\|_F$.
\medskip

It is proved in \cite[Theorem 5]{GJR-13} that, for a $C^1$-smooth function $f$ defined on a connected Finsler manifold, the Lipschitz constant of $f$ coincides with the supremum of the norm of its derivative. In fact, the same proof of \cite[Theorem 5]{GJR-13} gives also the corresponding one-sided result:

\smallskip

\begin{proposition}[$\|f|_{S}=\|df|_{S, \infty}$]
Let $(\mathcal{X},F)$ be a connected Finsler manifold and $f:\mathcal{X}\to \mathbb{R}$ a $C^1$-smooth function. Then $$ \|f\|_{\mathrm{Lip}}=\|df\|_\infty:=\sup\{\|df(x)\|_F\::\:x\in \mathcal{X}\}\in [0,\infty],$$ where $$\|f\|_{\mathrm{Lip}}=\displaystyle\sup_{x\neq y}\frac{|f(x)-f(y)|}{d_F(x,y)}\quad\text{is the Lipschitz constant of } f.$$
Similarly,
$$ \|f|_{S}=\|df|_{S, \infty}:=\sup\{\|df(x)|_F\::\:x\in \mathcal{X}\}\in [0,\infty],$$
where $$\|f|_{S}=\displaystyle\sup_{x\neq y}\frac{f(y)-f(x)}{d_F(x,y)}\quad\text{ is the semi-Lipschitz constant of } f.$$
\end{proposition}

\smallskip
As a direct consequence we obtain the following alternative description of $SC_{1^{-}}^1(\mathcal{X})$:

\begin{corollary}[The convex partially ordered set $SC_{1^{-}}^1(\mathcal{X})$]\label{Slip-dif}
Let $(\mathcal{X},d_{\mathcal{X}})$ be a connected Finsler manifold. Then
$$SC_{1^{-}}^1(\mathcal{X})=\{f\in C^1(X,\mathbb{R})\::\: \|f|_S<1\}=\{f\in C^1(X,\mathbb{R})\::\: \|df|_{S, \infty}<1\} .$$
\end{corollary}


\smallskip
Using the above result, we can easily see that, in the case of compact manifolds, every almost isometry is strict.

\begin{proposition}[Almost isometries for compact Finsler manifolds]\label{compact}
Let $(\mathcal{X},F)$ and $(\mathcal{Y},G)$ be connected and compact Finsler manifolds. Then every almost isometry $\tau:\mathcal{X}\to \mathcal{Y}$ is strict.
\end{proposition}

\begin{proof}
Consider the function $\phi:\mathcal{X}\to \mathbb{R}$ associated to $\tau$ in the sense of Proposition~\ref{ai}. By the above proposition we have that $G=\tau_*(F) + d(\phi\circ \tau^{-1})$. Then for every $x\in \mathcal{X}$ and every $v\in T_x\mathcal{X}$:
$$
G( \tau (x), d\tau (x) (v)) = F(x, v) + d \phi(x)(v).
$$
As a consequence, if $F(x, v)=1$, since we have that $d\tau (x)(v) \neq 0$, and then $G( \tau (x), d\tau (x) (v))>0$, it follows that  $d\phi (x)(v)<1$. 

For every $x\in \mathcal{X}$, the indicatrix $S_x := \{ v\in T_x\mathcal{X} \, : \, F(x, v)=1 \}$ is compact. Therefore, for each fixed $x_0\in \mathcal{X}$  we can choose a compact neighborhood $W^{x_0}$ such that the portion of the indicatrix bundle over $W^{x_0}$ is a compact set. That is, the set
$$
B_{x_0} = \{(x, v)\in T\mathcal{X} \, : \, x\in W^{x_0}; \, v\in T_x\mathcal{X}, \, F(x, v)=1\}
$$
is compact, and furthermore $d\phi (x)(v)<1$ for every $(x, v) \in B_{x_0}$. Then $\|d\phi (x)|_S <1$ for every $x \in W^{x_0}$. Now, from the compactness of $\mathcal{X}$ we obtain that $\|d\phi|_{S, \infty} <1$. Then by Corollary~\ref{Slip-dif} we have that $\|\phi|_S<1$. Finally, considering $\tau^{-1}$ and using Proposition~\ref{sai} we obtain the result. 
\end{proof}

We next give a simple example of non-strict almost isometry:

\begin{example}[Nonstrict almost isometry]
 Let $\mathcal{X}=\mathcal{Y}= \mathbb R$. We consider on $\mathcal{X}$ the usual Finsler structure $F_{\mathcal{X}}(x, v) = \vert v \vert$ and we define on $\mathcal{Y}$ the Finsler structure $F_{\mathcal{Y}}(x, v) = \vert v \vert - d \phi(x)(v)$, where $\phi : \mathbb R \to \mathbb{R}$ is given by
$$
\phi (x) := \int_0^x \frac{t^2}{1+t^2} \, dt.
$$
Note that $(\mathcal{Y}, F_{\mathcal{Y}})$ is a Randers space, since $\vert \phi'(x) \vert < 1$ for every $x\in \mathbb{R}$. It is easy to see that the associated Finsler distances are $d_{\mathcal{X}} (x, x')= \vert x-x'\vert$ and $d_{\mathcal{Y}} (x, x')= \vert x-x'\vert + \phi (x) - \phi(x')$. In this way we obtain that the identity map $\tau : \mathcal{X} \to \mathcal{Y}$ given by $\tau (x)=x$ is an almost isometry from $(\mathcal{X}, d_{\mathcal{X}})$ to $(\mathcal{Y}, d_{\mathcal{Y}})$. Nevertheless in this case we have that $\|\phi|_S=1$. Therefore by Proposition~\ref{sai} the almost isometry $\tau$ is not strict.
\end{example}

\medskip

\noindent The following proposition shows that the elements of $SC_{1^-}^1(\mathcal{X})$ can be used to describe open sets of~$\mathcal{X}$. The proof is omitted, as it follows from standard smooth manifold arguments.
\begin{proposition}[co-zero sets]\label{cozero}
Let $(\mathcal{X},F)$ be a Finsler manifold and $\mathcal{U}$ an open subset of $\mathcal{X}$. Then, there exists a smooth function $f:\mathcal{X}\to[0,\infty)$ such that $$\mathcal{U}=\{x\in \mathcal{X}:\:f(x)>0\}.$$ Moreover, $f$ can be chosen so that $\|df\|_\infty< 1$, and therefore $f\in SC_{1^-}^1(\mathcal{X})$.
\end{proposition}

\medskip
Let us now recall from \cite[Theorem 8]{GJR-13} the following smooth approximation theorem. An adaptation of this result (stated below as Corollary~\ref{smoothslip}) will be one of the key elements of our main result.

\begin{theorem}[Smooth approximation of Lipschitz functions in Finsler manifolds]
Let $(\mathcal{X},F)$ be a connected, second countable Finsler manifold, ${f:\mathcal{X}\to \mathbb{R}}$ a Lipschitz function, $\varepsilon:\mathcal{X}\to (0,+\infty)$ a continuous function and $r>0$. Then, there exists a $C^1$-smooth Lipschitz function $g:\mathcal{X}\to \mathbb{R}$ such that:
\begin{enumerate}
\leftskip .35pc
    \item[$(i)$] $|g(x)-f(x)|\leq \varepsilon(x)$ for all $x\in \mathcal{X}$ ; \medskip
    \item[$(ii)$] $\|g\|_{\mathrm{Lip}}\leq \|f\|_{\mathrm{Lip}}+r$.
\end{enumerate}
\end{theorem}
\medskip
By replacing the Lipschitz functions by semi-Lipschitz functions in Proposition 6, Lemma 7 and Theorem 8 of \cite{GJR-13}, we obtain the following corollary:

\begin{corollary}[Smooth approximation of semi-Lipschitz functions in Finsler manifolds]\label{smoothslip}
Let $(\mathcal{X},F)$ be a connected, second countable Finsler manifold, $f:\mathcal{X}\to \mathbb{R}$ a semi-Lipschitz function, $\varepsilon:\mathcal{X}\to (0,+\infty)$ a continuous function and $r>0$. Then, there exists a $C^1$-smooth semi-Lipschitz function $g:\mathcal{X}\to \mathbb{R}$ that approximates $f$ in the following sense:
\begin{enumerate}
\leftskip .35pc
    \item[$(i)$] $|g(x)-f(x)|\leq \varepsilon(x)$ for all $x\in \mathcal{X}$ ; \medskip
    \item[$(ii)$] $\|g|_{S}\leq \|f|_{S}+r$.
\end{enumerate}
\end{corollary}
The proof of Corollary~\ref{smoothslip} (which is based to results analogous to Proposition 6 and Lemma 7 of \cite{GJR-13}) is omitted, since all arguments are straightforward adaptations of the aforementioned ones, by replacing Lipschitz bounds with semi-Lipschitz ones. \smallskip

The following proposition shows that given two connected Finsler manifolds $\mathcal{X}$ and $\mathcal{Y}$, each strict almost isometry between $\mathcal{X}$ and $\mathcal{Y}$ (with respect to their Finsler distances) induces an isomorphism of convex partially ordered sets between $SC_{1^-}^1(\mathcal{Y})$ and $SC_{1^-}^1(\mathcal{X})$.

\begin{proposition}[Strict almost isometries induce convex partially ordered isomorphsims]\label{isomorph}
Let $\mathcal{X}$, $\mathcal{Y}$ be connected Finsler manifolds and $\tau:\mathcal{X}\to \mathcal{Y}$ a strict almost isometry with respect to their Finsler distances induced by a function $\phi:X\to \mathbb{R}$ (in the sense of Proposition~\ref{ai}). Then the mapping
\[
\left \{
\begin{array}
[l]{l}
T:SC_{1^-}^1(\mathcal{Y})\to SC_{1^-}^1(\mathcal{X}) \medskip \\
Tf=f\circ \tau + \phi \medskip
\end{array}
\right.
\]
is an isomorphism of convex partially ordered sets.
\end{proposition}
\begin{proof}
Consider the mapping $Tf=f\circ \tau + \phi$. Note that the convexity and order-preserving properties of $T$ are immediate, so we only need to check that $T$ is a well-defined bijection. To this end, note first that if $\|f|_S\leq 1$, then $\|Tf|_S\leq 1$, since:
	$$Tf(x')-Tf(x)=f(\tau(x'))-f(\tau(x)) +\phi(x') -\phi(x)\leq d_{\mathcal{Y}}(\tau(x),\tau(x'))+\phi(x') -\phi(x)=d_{\mathcal{X}}(x,x').$$
\smallskip
We shall now prove that if $f\in SC_{1^-}^1(\mathcal{Y})$ then $\|Tf|_S < 1$. Note that $T0= \phi$ and from Proposition~\ref{sai} we have that $\|\phi|_S<1$. Choose $\lambda \in (0,1)$ such that  $\|{\lambda}^{-1} f |_S < 1$. Then
	\begin{align}
	\|Tf|_S&=\left\|T\left((\lambda{\lambda}^{-1})f +(1-\lambda)0\right)\right|_S=\left\|\lambda T\left({\lambda}^{-1}f \right)+\left(1-\lambda\right)T0\right|_S  \\ \nonumber
	&\leq \lambda\left\|T\left({\lambda}^{-1}f\right)\right|_S+\left(1-\lambda\right)\left\|T0\right|_S \leq \lambda + (1-\lambda) \|T0|_S <1.
	\end{align}
	This shows that $T\left(SC_{1^-}^1(\mathcal{Y})\right) \subset SC_{1^-}^1(\mathcal{X})$ and $T$ is well-defined. An analogous argument holds for the inverse mapping $T^{-1}g=g\circ\tau^{-1} -\phi\circ\tau^{-1}$, so we conclude that $T$ is a bijection.
\end{proof}
\medskip

\section{Main result}\label{main}
The main result of this work is the converse of Proposition~\ref{isomorph} which eventually provides a functional characterization of strict almost isometries between connected, second countable and bicomplete Finsler manifolds, which becomes a characterization of all almost isometries in the compact setting (see forthcoming Corollaries~\ref{strictalmostiso}--\ref{compactalmost}).
\smallskip
\begin{theorem}[Main result]\label{teo}
Let $(\mathcal{X},d_{\mathcal{X}})$ and $(\mathcal{Y},d_{\mathcal{Y}})$ connected, second countable Finsler manifolds which are bicomplete (with their respective Finsler distances). Assume there exists an isomorphism of convex partially ordered sets ${T:SC_{1^-}^1(\mathcal{Y})\to SC_{1^-}^1(\mathcal{X})}$. Then, there exist $\alpha>0$, a quasi-metric $d'_X$ on $X$ and a bijection $\tau:\mathcal{X}\to \mathcal{Y}$ such that: \smallskip
\begin{enumerate}
\leftskip .25pc
    \item[$(i)$] $(\mathcal{X}, d_{\mathcal{X}}$ is almost isometric to $(\mathcal{X},d'_{\mathcal{X}})$.\smallskip
    \item[$(ii)$] $(\mathcal{X},\alpha\cdot d_{\mathcal{X}}')$ is isometric to $(\mathcal{Y},d_{\mathcal{Y}})$ via $\tau$.\smallskip
    \item[$(iii)$] $\mathcal{X}$ is diffeomorphic to $\mathcal{Y}$ via $\tau$.\smallskip
    \item[$(iv)$] $\forall f\in SC_{1^-}^1(\mathcal{Y}),\:Tf=c\cdot (f\circ \tau)+\phi$, with $c={\alpha}^{-1}$ and $\phi=T0$.\\ (In particular, $\phi$ is smooth and $\|\phi|_S<1$).
\end{enumerate}
\end{theorem}
The proof of the above theorem will be given in Subsection~\ref{ss-3.3}. Before, we shall need to establish several intermediate results. In what follows, we shall always assume that $(\mathcal{X},d_{\mathcal{X}})$, $(\mathcal{Y},d_{\mathcal{Y}})$ are connected, second countable and bicomplete Finsler manifolds and $T$ will denote an isomorphism of the convex partially ordered sets $SC_{1^-}^1(\mathcal{Y})$ and $SC_{1^-}^1(\mathcal{X})$.

\medskip
\subsection{Basis for the topologies}
The following definition introduces some useful notation and describes a certain type of open subsets of the Finsler manifolds that are naturally associated with the class of smooth semi-Lipschitz functions.
\begin{definition}[Open sets related to the order structure]
Let $h\in SC_{1^-}^1(\mathcal{Y})$. We define
\begin{align*}
    SC_{1^-}^1(\mathcal{Y})_h=&\{f\in SC_{1^-}^1(\mathcal{Y}):\:f\geq h\},\\
    SC_{1^-}^1(\mathcal{X})_{Th}=&\{g\in SC_{1^-}^1(\mathcal{X}):\:g\geq Th\}=T(SC_{1^-}^1(\mathcal{Y})_h).
\end{align*}
Furthermore, for any $f\in SC_{1^-}^1(\mathcal{Y})_h$, we denote:
\begin{align*}
    \mathrm{supp}_h(f)=&\overline{\{y\in \mathcal{Y}:\:f(y)>h(y)\}} \qquad \text{and}\quad   \mathcal{V}_h^f=\mathrm{int}\left(\mathrm{supp}_h(f)\right), \smallskip \\
    \mathrm{supp}_{Th}(Tf)=&\overline{\{x\in \mathcal{X}:\:Tf(x)>Th(x)\}}\quad\text{and}\quad
    \mathcal{U}_{Th}^{Tf}=\mathrm{int}\left(\mathrm{supp}_{Th}(Tf)\right), \smallskip
\end{align*}
\noindent where closure and interior are taken in the symmetric topologies of $(\mathcal{Y},d_{\mathcal{Y}})$ and $(\mathcal{X},d_{\mathcal{X}})$.
\end{definition}
Before we proceed, let us introduce the notion of \emph{bump} function on a Finsler manifold $\mathcal{X}$.
\begin{definition}[(Smooth semi-Lipschitz) bump functions]  \label{bump}
Let $\mathcal{X}$ be Finsler manifold. A nonnegative smooth semi-Lipschitz function $b:\mathcal{X}\rightarrow \mathbb{R}_+$ is called a bump function on $\mathcal{X}$ centered at a point $x_0 \in \mathcal{X}$, provided $b(x_0)>0$ and  $\mathrm{supp}(b)\subset B_{\mathcal{X}}(x_0,r)$ for some $r>0$.
\end{definition}
It is well-known that for every $x_0\in\mathcal{X}$ and $r>0$ there exist a bump function $b\in SC_{1^{-}}^1(\mathcal{X})_0 $ with $\mathrm{supp}(b)\subset B_{\mathcal{X}}(x_0,r)$ and $b(x_0)>0$. \smallskip

We are now ready to describe a basis for the topologies in $\mathcal{Y}$ and $\mathcal{X}$ respectively, which will play an important role in the sequel.
\begin{proposition}[Topology basis for $\mathcal{X}$ and $\mathcal{Y}$]\label{base}
	Let $\mathcal{X}, \mathcal{Y}$ two Finsler manifolds and let us fix a function $h\in SC_{1^-}^1(\mathcal{Y})$. Then the families
	\[
\mathcal{B}_h(\mathcal{Y})=\{\mathcal{V}_h^f:~f\in SC_{1^-}^1(\mathcal{Y})_h\}\,\,\textrm{ and }\, \,   \mathcal{B}_h(\mathcal{X})=\{\mathcal{U}_{Th}^{Tf}:~f\in SC_{1^-}^1(\mathcal{Y})_h\}
	\]
are basis for the topologies of $(\mathcal{Y},d_{\mathcal{Y}})$ and $(\mathcal{X},d_{\mathcal{X}})$ respectively.
	\end{proposition}
	\begin{proof}
Given $y_0\in \mathcal{Y}$ and a ball $B_{\mathcal{Y}}:=B_{\mathcal{Y}}(y_0,r)$ for the distance $d_{\mathcal{Y}}^s$ centered at $y_0$ and of radius $r>0$, we take a bump function ${b\in SC_{1^-}^1(\mathcal{Y})_0}$ such that ${\mathrm{supp}(b)\subset B_Y}$, $b(y_0)>0$ and $\|b|_S+\|h|_S< 1$. Defining $f=h+b$, we get that $y_0\in \mathcal{V}_h^f \subset B_{\mathcal{Y}}$.  \smallskip

Given $x_0\in \mathcal{X}$ and a ball $B_{\mathcal{X}}$ for $d_{\mathcal{X}}^s$ containing $x_0$, take $b\in SC_{1^-}^1(\mathcal{X})_0$ such that $\mathrm{supp}(b)\subset B_{\mathcal{X}}$, $b(x_0)>0$ and $\|b|_S+\|Th|_S<1$. Since $Th+b\geq Th$ and $T$ is an isomorphism of convex partially ordered sets, there exists $f\in SC_{1^-}^1(\mathcal{Y})_h$ such that $Tf=Th+b$. Therefore, $x_0\in \mathcal{U}_{Th}^{Tf}\subset B_{\mathcal{X}}$.
	\end{proof}

	\medskip
There is a natural bijection between the basis $\mathcal{B}_h(\mathcal{Y})$ and $\mathcal{B}_h(\mathcal{X})$:
\begin{definition}[Bijection between $\mathcal{B}_h(\mathcal{Y})$ and $\mathcal{B}_h(\mathcal{X})$ induced by $T$] \label{I}
			Let $h\in SC_{1^-}^1(\mathcal{Y})$. Then the mapping $\mathcal{I}_h:\mathcal{B}_{h}(\mathcal{Y})\to \mathcal{B}_h(\mathcal{X})$ defined by $T(\mathcal{V}_h^f)= \mathcal{U}_{Th}^{Tf}$ is a bijection.
\end{definition}
	\medskip
	\begin{remark}
		The aforementioned basis appear to depend on the choice of the function $h$. Nonetheless, we shall show in forthcoming Proposition~\ref{h} and respectively, Corollary~\ref{cor-ind}, that the basis $\mathcal{B}_h(\mathcal{X})$, $\mathcal{B}_h(\mathcal{Y})$ and, respectively, the bijection $\mathcal{I}_h$ do not depend on the choice of $h$.
	\end{remark}
\medskip
	Next, we show that for each $h\in SC_{1^-}^1(\mathcal{Y})$, the bijection $\mathcal{I}_h$ preserves the order structure of $(\mathcal{B}_{h}(\mathcal{Y}),\subset)$ and $(\mathcal{B}_h(\mathcal{X}),\subset)$. To this end, following \cite{CC} we introduce the following notation:\medskip
	\begin{enumerate}[\rm (i)]
		\item $f\sqcap_h g=\{u\in  SC_{1^-}^1(\mathcal{Y})_h:u\leq f,~u\leq g\}$. \smallskip
		\item $f\sqsubset_h g$ if for any $u\in  SC_{1^-}^1(\mathcal{Y})_h,~u\sqcap_h g=\{h\} \implies u\sqcap_h f=\{h\}$.\smallskip
		\item $Tf\sqcap_{Th} Tg=\{v\in  SC_{1^-}^1(\mathcal{X})_{Th}:v\leq Tf,~v\leq Tg\}$.\smallskip
		\item $Tf\sqsubset_{Th} Tg$ if for any $v\in  SC_{1^-}^1(\mathcal{X})_{Th},~v\sqcap_{Th} Tg=\{Th\} \implies v\sqcap_{Th} Tf=\{Th\}$.\smallskip
	\end{enumerate}

The following proposition gives more insight to the above notation. The proof follows the ideas of \cite{CC}.

\begin{proposition}\label{inc} Let $h\in SC_{1^-}^1(\mathcal{Y})$ and $f,g\in SC_{1^-}^1(\mathcal{Y})_h$. Then
		\begin{enumerate}[\rm(i)]
			\item $f\sqcap_h g=\{h\}\iff \mathcal{V}_h^f\cap \mathcal{V}_h^g=\emptyset$.\smallskip
			\item $f\sqsubset_h g \iff \mathcal{V}_h^f\subset \mathcal{V}_h^g \iff \mathrm{supp}_h(f)\subset \mathrm{supp}_h(g)$.\smallskip
			\item $Tf\sqcap_{Th} Tg=\{Th\}\iff \mathcal{U}_{Th}^{Tf}\cap \mathcal{U}_{Th}^{Tg}=\emptyset$.\smallskip
			\item $Tf\sqsubset_{Th} Tg \iff \mathcal{U}_{Th}^{Tf}\subset \mathcal{U}_{Th}^{Tg} \iff \mathrm{supp}_{Th}(Tf)\subset \mathrm{supp}_{Th}(Tg)$.
		\end{enumerate}
	\end{proposition}
	\begin{proof}
	$(i)$ If $\mathcal{V}_h^f\cap \mathcal{V}_h^g =\emptyset$ and $u\in f\sqcap_h g$, then ${u(y)\leq f(y)\land g(y)}$ for all $y\in ({\mathcal{V}}_h^f)^c\cup ({\mathcal{V}}_h^g)^c=\mathcal{Y}$, so $u=h$.\\
			Conversely, suppose that ${\mathcal{V}}_h^f\cap {\mathcal{V}}_h^g\neq \emptyset$ and let $y\in {\mathcal{V}}_h^f\cap {\mathcal{V}}_h^g$. Since $y\in \mathrm{supp}_h (f),$ there exists a sequence $\{y_n\}\subset \{y:~f(y)>h(y)\}$, with $y_n\to y$. Then, there exists $N\in \mathbb{N}$ such that $y_N\in {\mathcal{V}}_h^f\cap {\mathcal{V}}_h^g$ and $f(y_N)>h(y_N)$. Since $y_N\in \mathrm{supp}_h(g),$ there exists a sequence $(y^j)\subset \{y:~ g(y)>h(y)\}$ such that $y^j\to y_N$, and there is $J\in \mathbb{N}$ such that $y^J\in (f-h)^{-1}(0,\infty)\cap {\mathcal{V}}_h^f\cap {\mathcal{V}}_h^g$, so $f(y^J)\lor g(y^J)>h(y^J)$. Taking a suitable bump function $b\in SC_{1^-}^1(\mathcal{Y})_0$ with positive value at $y^J$, we get that $h\lneq b+h \in f\sqcap_h g$, and therefore $f\sqcap_h g\neq\{h\}$.\smallskip

\noindent $(ii)$ Let $f,g\in SC_{1^-}^1(\mathcal{Y})_h$. Then we have
						\begin{align} f\sqsubset_h g\iff& \left[\forall u\in SC_{1^-}^1(\mathcal{Y})_h,~u\sqcap_hg=\{h\}\implies u\sqcap_h f=\{h\}\right],\nonumber\\
			\iff& \left[\forall u\in SC_{1^-}^1(\mathcal{Y})_h,~ {\mathcal{V}}_h^u\cap {\mathcal{V}}_h^g=\emptyset\implies {\mathcal{V}}_h^u\cap {\mathcal{V}}_h^f=\emptyset\right],\nonumber\\
			\iff& \left[\forall u\in SC_{1^-}^1(\mathcal{Y})_h,~ {\mathcal{V}}_h^g\subset \left({\mathcal{V}}_h^u\right)^c \implies {\mathcal{V}}_h^f \subset \left({\mathcal{V}}_h^u\right)^c\right].\label{eqn}
			\end{align}
			Clearly, ${\mathcal{V}}_h^f \subset {\mathcal{V}}_h^g$ implies $f\sqsubset_h g$. On the other hand, since $\mathcal{B}_h(\mathcal{Y})$ is a basis, we can express $\overline{{\mathcal{V}}_h^g}$ as an intersection of sets of the form $\left({\mathcal{V}}_h^u\right)^c$. Let $\mathcal{H}\subset SC_{1^-}^1(\mathcal{Y})_h$ such that $\overline{{\mathcal{V}}_h^g}=\displaystyle\bigcap_{u\in \mathcal{H}}\left({\mathcal{V}}_h^u\right)^c $. Then, using \eqref{eqn} we deduce:
			$$	f\sqsubset_h g \implies~\forall u\in \mathcal{H}~{\mathcal{V}}_h^g\subset \left({\mathcal{V}}_h^u\right)^c
						   \implies {\mathcal{V}}_h^f \subset \left({\mathcal{V}}_h^u\right)^c,$$
so that
			$${\mathcal{V}}_h^f\subset \overline{{\mathcal{V}}_h^g} \implies \overline{{\mathcal{V}}_h^f}\subset \overline{{\mathcal{V}}_h^g}\implies {\mathcal{V}}_h^f\subset {\mathcal{V}}_h^g,$$ since the elements of $\mathcal{B}_h(\mathcal{Y})$ are regular open sets (that is, they coincide with the interior of their closure). \smallskip
		
\noindent The proofs of $(iii)$ and $(iv)$ are analogous to the proofs of $(i)$ and $(ii)$ respectively, and will be omitted.
	\end{proof}
	\smallskip
	We have shown that inclusions between members of $\mathcal{B}_h(\mathcal{Y})$ (and $\mathcal{B}_h(\mathcal{X})$) can be described using the relation $\sqsubset_h$ on $SC_{1^-}^1(\mathcal{Y})_h$ (respectively $\sqsubset_{Th}$ on $SC_{1^-}^1(\mathcal{X})_h$), which depends only on the convex and order structure of $SC_{1^-}^1(\mathcal{Y})$ (respectively $SC_{1^-}^1(\mathcal{X})$), so we can use the isomorphism $T$ to relate inclusions between sets of each basis.
	\begin{proposition}
 Let $h\in SC_{1^-}^1(\mathcal{Y})$ and $f,g\in SC_{1^-}^1(\mathcal{Y})_h$. Then
		$$f\sqsubset_h g \iff Tf\sqsubset_{Th} Tg.$$
		Therefore, $${\mathcal{V}}_h^f\subset {\mathcal{V}}_h^g \iff \mathcal{U}_{Th}^{Tf}\subset \mathcal{U}_{Th}^{Tg}.$$
			\end{proposition}
	\begin{proof}
	Let $f,g\in SC_{1^-}^1(\mathcal{Y})_h$. Using the properties of $T$, we obtain:
		\begin{align*}
		f\sqsubset_h g \iff& \left[ \forall u\in SC_{1^-}^1(\mathcal{Y})_h,~u\sqcap_h g=\{h\}\implies u\sqcap_h f=\{h\}\right],\\
		\iff& \left[ \begin{array}{l}
		     \forall u\in SC_{1^-}^1(\mathcal{Y})_h,~\{v\in SC_{1^-}^1(\mathcal{Y})_h:~v\leq u\land g\}=\{h\} \\\implies
		  \{v\in SC_{1^-}^1(\mathcal{Y})_h:~v\leq u\land f\}=\{h\}
		\end{array}\right],\\
		\iff& \left[\begin{array}{l}
		    \forall Tu\in SC_{1^-}^1(\mathcal{X})_{Th},~\{Tv\in SC_{1^-}^1(\mathcal{X})_{Th}:~Tv\leq Tu\land Tg\}=\{Th\} \\ \implies \{Tv\in SC_{1^-}^1(\mathcal{X})_{Th}:~Tv\leq Tu\land Tf\}=\{Th\}
		\end{array}\right],\\
		\iff&~ Tf\sqsubset_{Th} Tg.
		\end{align*}
		The second part of the statement follows directly using Proposition~\ref{inc}.
	\end{proof}
	\medskip
	\begin{corollary} For any $h\in SC_{1^-}^1(\mathcal{Y})$, the mapping $\mathcal{I}_h$ from Definition~\ref{I} is an order-preserving bijection, that is, for any $\mathcal{V}_1,\mathcal{V}_2\in \mathcal{B}_h(\mathcal{Y})$ $$\mathcal{V}_1\subset \mathcal{V}_2 \iff \mathcal{I}_h\left(\mathcal{V}_1\right)\subset\mathcal{I}_h(\mathcal{V}_2).$$
\end{corollary}
\smallskip
Next, we show that local inequalities between elements of $SC_{1^-}^1(\mathcal{Y})$ can be characterized via its convex-order structure. In what follows, we fix $h\in SC_{1^-}^1(\mathcal{Y})$ and for any $g\in SC_{1^-}^1(\mathcal{Y})_h$ and $\lambda\in [0,1]$, we set: $$g_\lambda := {\lambda g + (1-\lambda)h}.$$	
	
\begin{proposition}[Characterization of dominance on $\mathcal{V}_h^f$]\label{mayor}
		Let $h\in SC_{1^-}^1(\mathcal{Y})$ and $\varphi, \psi, f\in SC_{1^-}^1(\mathcal{Y})_h$. Then we have:
		$$\varphi \geq \psi \text{ on } {\mathcal{V}}_h^f \iff \forall \lambda\in [0,1],~\forall u\sqsubset_h f,~ \psi_\lambda \sqcap_h u \subset \varphi_\lambda\sqcap_h u.$$
	\end{proposition}

	\begin{proof}
	The ``only if'' implication is straightforward. For the ``if'' implication, suppose there is $y_0\in {\mathcal{V}}_h^f$ such that ${\psi(y_0)>\varphi(y_0)}$. Then there is a symmetric ball $B$  containing $y_0$ such that ${\psi >\varphi}$ on $B$. Furthermore, we can take $u\in SC_{1^-}^1(\mathcal{Y})_h$ defined by $u=h+b$, with $b:\mathcal{Y}\to [0,\varepsilon]$ a $C^1$ Lipschitz bump function supported on $B$ such that  $b(y_0)=\varepsilon$, for some $\varepsilon>0$. With this, $u(y_0):=\alpha>h(y_0)$ and $y_0\in \overline{{\mathcal{V}}_h^u} \subset B \subset {\mathcal{V}}_h^f$. Without lost of generality, $\psi(y_0)> \alpha$. Let $\lambda\in [0,1]$ such that $\psi_\lambda (y_0)> \alpha > \varphi_\lambda(y_0)$, and let $\eta:Y\to [0,1]$ be a $C^1$ Lipschitz bump function such that $\eta|_{\overline{\{\psi_\lambda < u\}}} =0$, $\eta(y_0)=1$, and define $v=\eta b+h$. Note that $v$ is semi-Lipschitz, since for $y,y'\in \mathcal{Y}$ we have:
	\begin{align*}
	v(y')-v(y)&=\eta(y')b(y')+h(y')-\eta(y)b(y) -h(y)\\
	&=\eta(y')b(y') -\eta(y')b(y) +\eta(y')b(y)-\eta(y)b(y)+h(y')-h(y)\\
	&\leq \|\eta\|_\infty(b(y')-b(y))+\|b\|_\infty(\eta(y')-\eta(y))+\|h|_S\,{d_{\mathcal{Y}}(y,y')}\\
	&\leq (\|b\|_{\mathrm{Lip}}+ \varepsilon\|\eta\|_{\mathrm{Lip}}+\|h|_S)\,d_{\mathcal{Y}}(y,y').
	\end{align*}
	Choose $t\in[0,1]$ such that $v_t\in SC_{1^-}^1(\mathcal{Y})$. Since for $g\in SC_{1^-}^1(\mathcal{Y})$ and $\lambda\in [0,1]$,
	$$(g_\lambda)_t=tg_\lambda +(1-t)h=\lambda tg+(1-\lambda t)h=g_{\lambda t},$$
	we get that $v_t\in SC_{1^-}^1(\mathcal{Y})_h$, $v_t\sqsubset_h f$ (since $\mathrm{supp}_h(v_t)\subset \mathrm{supp}(b)\subset B\subset \mathrm{supp}_h(f)$), $v_t(y)\leq \psi_{\lambda t}(y)$ for all $y\in Y$, and finally $v_t(y_0)=u_t(y_0)>\varphi_{\lambda t}(y_0)$. Therefore, $v_t\in \left(\psi_{\lambda t}\sqcap_h v_t \right)\setminus \left(\varphi_{\lambda t}\sqcap_h v_t\right)$, a contradiction.
\end{proof}

We now state the following useful lemma.
\begin{lemma}[Transfer principle] \label{lema1}
		Let $h\in SC_{1^-}^1(\mathcal{Y})$ and $\varphi, \psi, f\in SC_{1^-}^1(\mathcal{Y})_h$. Then:
		$$\varphi \geq \psi \text{ on } {\mathcal{V}}_h^f \iff T\varphi \geq T\psi \text{ on } \mathcal{U}_{Th}^{Tf}.$$
\end{lemma}
	
\begin{proof} It follows from Proposition~\ref{mayor}, since the right side of the equivalence depends only on the convex and order structure of $SC_{1^-}^1(\mathcal{Y})$, which is preserved by $T$, so for any $u,v,f,g\in SC_{1^-}^1(\mathcal{Y})_h$ we have:
$${u\sqsubset_h f} \iff Tu\sqsubset_h Tf \quad \text{and}\quad v\in f\sqcap_h g \iff Tv\in Tf\sqcap_h Tg.$$
Therefore $(T\varphi)_\lambda=T(\varphi_\lambda),$ for any $\varphi\in SC_{1^-}^1(\mathcal{Y})_h$ and $\lambda\in [0,1]$.
	\end{proof}
	\medskip
	Next, we show that the basis $\mathcal{B}_h(\mathcal{Y})$, $ \mathcal{B}_h(\mathcal{X})$ and the bijection $\mathcal{I}_h$ are independent of $h$.
\begin{proposition}[Independence of the topological basis from $h$]\label{h}
        Let $h\in SC_{1^-}^1(\mathcal{Y})$. Then \smallskip \newline
$\rm(i).\,$ $\mathcal{B}_h(\mathcal{Y})=\mathcal{B}_0(\mathcal{Y}):=\mathcal{B}(\mathcal{Y})$ \smallskip \\
$\rm(ii).$ $\mathcal{B}_h(\mathcal{X})=\mathcal{B}_0(\mathcal{X}):=\mathcal{B}(\mathcal{X}).$
 \end{proposition}

\begin{proof}
 $\rm(i).$ Let $\mathcal{V}_0^f \in \mathcal{B}(\mathcal{X})$. Since $\|h|_S<1$, there is $\lambda\in (0,1]$ such that $\lambda f+h\in SC_{1^-}^1(\mathcal{Y})$, so $\{f>0\}=\{\lambda f+h>h\}$, and therefore $\mathcal{V}_0^f=\mathcal{V}_h^{\lambda f+h}\in \mathcal{B}_h(\mathcal{Y})$.  Conversely, let ${\mathcal{V}}_h^g\in \mathcal{B}_h(\mathcal{Y})$. Since the set $\{g>h\}:=\{y\in\mathcal{Y}:\,g(y)>h(y)\}$ is open, by Proposition~\ref{cozero} there exists $f\in SC_{1^-}^1(\mathcal{Y})_0$ such that $\{f>0\}=\{g>h\}$, so ${\mathcal{V}}_h^g=\mathcal{V}_0^f\in \mathcal{B}(\mathcal{Y})$.\smallskip

\noindent $\rm(ii).$ Let $\mathcal{U}_{T0}^{Tf} \in \mathcal{B}(X)$, and $g\in SC_{1^-}^1(\mathcal{X})_0$ such that ${\{Tf>T0\}=\{g>0\}}$. Take $\lambda\in (0,1]$ such that $\lambda g +Th\in SC_{1^-}^1(\mathcal{X})$, and since $\lambda g +Th\geq Th$, there exists $\tilde{f}\in SC_{1^-}^1(\mathcal{Y})_h$ such that $T\tilde{f}=\lambda g + Th$. Therefore, $\{Tf>T0\}=\{T\tilde{f}>Th\}$ and $\mathcal{U}_{T0}^{Tf}=\mathcal{U}_{Th}^{T\tilde{f}}\in \mathcal{B}_h(\mathcal{X})$. Conversely, let $\mathcal{U}_{Th}^{Tf}\in \mathcal{B}_h(\mathcal{X})$ and $g\in SC_{1^-}^1(\mathcal{X})_0$ such that $\{Tf>Th\}=\{g>0\}$. Taking $\lambda\in (0,1]$ such that $\lambda g +T0\in SC_{1^-}^1(\mathcal{X})$, we get that $\mathcal{U}_{Th}^{Tf}=\mathcal{U}_{T0}^{\lambda g+T0}\in \mathcal{B}(\mathcal{X})$.
	\end{proof}
The following result completes the transfer principle of Lemma~\ref{lema1}:
\begin{proposition}\label{prop}
     Let $h\in SC_{1^-}^1(\mathcal{Y})$ and $U\in \mathcal{B}(\mathcal{X})$. Then $\mathcal{V}=\mathcal{I}_h^{-1}(\mathcal{U})$ is the only element in $\mathcal{B}(\mathcal{Y})$ such that for any ${\varphi,\psi\in SC_{1^-}^1(\mathcal{Y})_h}$
    \begin{eqnarray}\label{eqn2}
    \varphi\geq \psi \text{ on }\mathcal{V} \iff T\varphi \geq T\psi \text{ on } \mathcal{U}.
    \end{eqnarray}
\end{proposition}	
\begin{proof}
Lemma~\ref{lema1} ensures us that $\mathcal{V}$ satisfies \eqref{eqn2}. Let $\tilde{\mathcal{V}}\neq \mathcal{V}$ in $\mathcal{B}(\mathcal{Y})$ satisfying the same property. Without loss of generality, $\mathcal{V}\setminus \tilde{\mathcal{V}}\neq \emptyset$. Since both sets are regular open sets, there exists $y\in \mathcal{V}\setminus\tilde{\mathcal{V}}$ and $\varepsilon>0$ such that the symmetric ball  $B(y,\varepsilon):=B$ is contained in $\mathcal{V}\setminus\tilde{\mathcal{V}}$. Given $y_1\neq y_2$ in $B$, we can take $\varphi,\psi \in SC_{1^-}^1(\mathcal{Y})_h$ such that $\mathrm{supp}_h(\varphi)\cup\mathrm{supp}_h(\psi)\subset B$, $\varphi(y_1)<\psi(y_1)$ and $\psi(y_2)<\varphi(y_2)$. Therefore $\varphi \ngeq \psi$ on $B\subset \mathcal{V}$, but $\varphi \geq \psi$ on $\tilde{\mathcal{V}}$, which contradicts \eqref{eqn2}.
\end{proof}	
\medskip
Using the above, we show the independence of the bijection (\textit{c.f.} Definition~\ref{I}) from $h$ for a particular case. (The general case will be given in Corollary~\ref{cor-ind}.)
\begin{proposition}\label{prop1}
 Let $h_1,h_2\in SC_{1^-}^1(\mathcal{Y})$, such that $~h_1\leq h_2$. Then $\mathcal{I}_{h_1}=\mathcal{I}_{h_2}$.
\end{proposition}
\begin{proof}
Let $\mathcal{U}\in \mathcal{B}(X)$, and let ${\varphi, \psi \in SC_{1^-}^1(\mathcal{Y})_{h_2}}$ such that $\varphi \geq \psi$ on $\mathcal{V}_2:=\mathcal{I}_{h_2}^{-1}(\mathcal{U})$. By Lemma~\ref{lema1}, $T\varphi\geq T\psi$ on $\mathcal{U}$, and since  $h_1\leq h_2$, $\varphi, \psi \in SC_{1^-}^1(\mathcal{Y})_{h_1}$, and by Lemma~\ref{lema1} $\varphi \geq \psi $ on $\mathcal{V}_1:=\mathcal{I}_{h_1}^{-1}(\mathcal{U})$. Hence, by Proposition~\ref{prop}, $\mathcal{V}_1=\mathcal{V}_2$, and therefore, $\mathcal{I}_{h_1}=\mathcal{I}_{h_2}$.
\end{proof}
\medskip
\begin{corollary}[Independence of the bijection from $h$] \label{cor-ind}
 Let $h\in SC_{1^-}^1(\mathcal{Y})$. Then
 $$\mathcal{I}_h=\mathcal{I}_0:=\mathcal{I}.$$
\end{corollary}
\begin{proof}
Consider $h\lor 0\in\mathrm{SLip}_1(\mathcal{Y})$ and note that $\|h\lor 0|_S\leq \|h|_S<1$. Take $\eta>0$ such that $\|h|_S+\eta< 1$ and $g:\mathcal{Y}\to \mathbb{R}$ a semi-Lipschitz $C^1$-smooth approximation given by Corollary~\ref{smoothslip}, using $\varepsilon=\frac{\eta}{2}$ and $r=\eta$. Replacing $g$ by $g+\varepsilon$ we get an approximation from above of $h\lor 0$, that is:
$$g\geq h\lor 0, \quad \|g|_S\leq \|h\lor 0|_S +\eta< 1 \quad \text{and }\,  \,\,g(y)-(h\lor 0)(y)\leq \eta, \,\,\forall y\in \mathcal{Y}.$$
It follows that $g\in SC_{1^-}^1(\mathcal{Y})$, $g\geq h$ and $g\geq 0$. By Proposition~\ref{prop1}, $\mathcal{I}_h=\mathcal{I}_g=\mathcal{I}_0$.
\end{proof}
\medskip
Thanks to this result, we can simply work with the basis $\mathcal{B}(\mathcal{Y})$ and $\mathcal{B}(\mathcal{X})$ (without fixing a function $h$) and with the bijection $\mathcal{I}:\mathcal{B}(\mathcal{Y})\to \mathcal{B}(\mathcal{X})$. The following lemma, established in \cite[Lemma 6]{CC11} is paramount for our considerations.

\begin{lemma}[Key Lemma] \label{lema6} Let $(X,d_{X})$ and $(Y,d_Y)$ be complete metric spaces, and let $B(X)$ and $B(Y)$ be basis for their topologies. If $\mathcal{I}: B(Y)\to B(X)$ is a inclusion-preserving bijection, then there exist dense subsets $X'\subset X$, $Y'\subset Y$ and an homeomorphism $\tau:X'\to Y'$ such that for every $x\in X'$ and $\mathcal{V}\in B(Y)$ it holds:
$$\tau(x)\in \mathcal{V} \iff x\in \mathcal{I}(\mathcal{V}).$$
\end{lemma}
\medskip
Since we deal with Finsler manifolds $\mathcal{X}$ and $\mathcal{Y}$ which are bicomplete, we can apply Lemma~\ref{lema6} to the underlying complete metric spaces $(\mathcal{X},d_{\mathcal{X}}^s)$ and $(\mathcal{Y},d_{\mathcal{Y}}^s)$ to obtain:
\begin{corollary}[Homeomorphism of dense subsets]\label{densos} Let $\mathcal{X}$, $\mathcal{Y}$ bicomplete Finsler manifolds. There exist dense subsets for the symmetrized topologies $\mathcal{X}'\subset \mathcal{X}$, $\mathcal{Y}'\subset \mathcal{Y}$  and an homeomorphism $\tau:\mathcal{X}'\to \mathcal{Y}'$ such that for any $x\in \mathcal{X}'$ and $\mathcal{V}\in\mathcal{B}(\mathcal{Y})$,
\begin{equation}\label{topo}
\tau(x)\in \mathcal{V}\iff x\in \mathcal{I}(\mathcal{V}).
\end{equation}
\end{corollary}
\smallskip
\subsection{Pointwise behaviour of the isomorphism of convex partially ordered sets}
The following result will allow us to deduce information about the pointwise behavior of the isomorphism $T$.

\begin{corollary}\label{cor1}
Let $f,g\in SC_{1^-}^1(\mathcal{Y})$, $\mathcal{X}'\subset \mathcal{X}$ the dense subset of Corollary~\ref{densos} and $x_0\in \mathcal{X}'$. Then, $$f(\tau(x_0))=g(\tau(x_0))\iff Tf(x_0)=Tg(x_0),$$
		\noindent where $\tau:\mathcal{X}'\to \mathcal{Y}'$ is the homeomorphism of Corollary~\ref{densos}.
		
\end{corollary}
\begin{proof}
We need to ensure that we can apply Lemma~\ref{lema1}. To this end, take $\varepsilon>0$ such that $\|f|_S\lor\|g|_S+\varepsilon<1$, and let $h$ be a $C^1$-smooth semi-Lipschitz approximation of $f\land g$ such that $h\leq f\land g$ and
 $\|h|_S\leq \|f|_S\lor\|g|_S+\varepsilon<1$. Set $y_0=\tau(x_0)$. It suffices to prove that $Tf(x_0)>Tg(x_0)$ implies $f(y_0)>g(y_0)$. Suppose ${Tf(x_0)>Tg(x_0)}$ and $f(y_0)\leq g(y_0)$. As the relation $Tf> Tg$ is satisfied on a neighborhood of $x_0$, the relation $f\geq g$ is satisfied on a neighborhood of $y_0$. Therefore, $f(y_0)=g(y_0)$ and $y_0$ is a local minimum of the function $f-g$, so $df(y_0)=dg(y_0).$ Let $\varphi\in SC_{1^-}^1(\mathcal{Y})_h$ such that $\varphi(y_0)=f(y_0)$ and $d\varphi(y_0)\neq df(y_0)$. Every neighborhood of $y_0$ contains points (and basic open sets) where $\varphi>f$ and where $\varphi<g$, and both types of points can be taken on the dense set $\mathcal{Y}'$. Taking sequences of these point converging to $y_0$ and applying Lemma~\ref{lema1} on the corresponding basic open sets satisfying the desired inequalities (and the respective basic neighborhoods of the preimages by $\tau$ of the elements of the sequences), it follows by continuity that  $Tf(x_0)\leq T\varphi(x_0)\leq Tg(x_0)$, a contradiction. The remaining implication follows by the same argument.
\end{proof}
\medskip

We shall now show that the convexity property of the isomorphism $T$ determine how its action on the constant functions.
\begin{proposition}[Action of $T$ on the constant functions]\label{ctes}
Let $g\in SC_{1^-}^1(\mathcal{Y})$. Then $Tg-T0$ is constant if and only if $g$ is constant. Moreover, there exists $\alpha>0$ such that
	$$T\lambda = T0+{\alpha}^{-1}{\lambda},\quad\forall \lambda\in \mathbb{R}.$$
\end{proposition}
	\begin{proof}
	Let $\lambda \in \mathbb{R}$ and $g^\lambda\in SC_{1^-}^1(\mathcal{Y})$ such that $Tg^\lambda=T0+\lambda$. \smallskip
	
Let us first assume that $\lambda\geq 1$. Then by convexity property of the isomorphism $T$ we deduce:
$$T\left({\lambda}^{-1}g^\lambda\right)=T\left({\lambda}^{-1}g^\lambda +{\lambda}^{-1}
(\lambda-1)0\right)={\lambda}^{-1}Tg^\lambda+{\lambda}^{-1}(\lambda-1)T0={\lambda}^{-1}T0+1+{\lambda}^{-1}(\lambda-1)T0=T0+1.$$
It follows that $Tg^1:=T0+1=T({\lambda}^{-1}g^{\lambda})$, therefore, since $T$ is bijective, $\lambda g^1=g^\lambda$ for all $\lambda\geq 1$, so $\|g^1|_S\leq {\lambda}^{-1}$ for all $\lambda\geq 1$. This latter yields that the function $g^1$ is constant, that is,
there exists $\alpha \in \mathbb{R}$ such that $g^1=\alpha$, whence $g^{\lambda}=\alpha \lambda$ for all $ \lambda\geq 1$. Since ${Tg^1=T0+1>T0}$, it follows that $\alpha >0$.\smallskip

Let us now consider the case $\lambda\in [0,1)$. Then $T\left(\lambda g^1\right)=\lambda Tg^1+(1-\lambda)T0=\lambda+T0=T\left(g^\lambda\right),$ therefore, $\lambda g^1=\lambda \alpha =g^\lambda$ for all $\lambda\in [0,1)$. It follows that $g^\lambda=\lambda\alpha$ for any $\lambda\geq 0$. In particular, $T\left(g^{{\alpha}^{-1}{\lambda}}\right)=T0+{\alpha}^{-1}{\lambda}{=T\lambda}$ for any $\lambda\geq 0$. \smallskip
	
Finally, using again convexity of $T$ we get:
$$ T0=T\left(\frac{1}{2}\lambda+\frac{1}{2}\left(-\lambda\right)\right)=\frac{1}{2}T\lambda+\frac{1}{2}T(-\lambda)=\frac{1}{2}\left(T0+{\alpha}^{-1}\lambda\right)+\frac{1}{2}T(-\lambda),$$
which yields $T(-\lambda)=T0-{\alpha}^{-1}\lambda$, for every $\lambda \geq 0$.
	\end{proof}
	\medskip
	Combining Proposition~\ref{ctes} and Corollary~\ref{cor1}, we obtain
	\begin{corollary}\label{cor}
	Let $f\in SC_{1^-}^1(\mathcal{Y})$, $\mathcal{X}'\subset \mathcal{X}$ the dense subset from Corollary~\ref{densos} and $x_0\in \mathcal{X}'$. Denoting by $c={\alpha}^{-1}=T1-T0$ and $\phi=T0$, we have that
	$$Tf(x_0)=c\cdot f(\tau(x_0))+\phi(x_0).$$
	\end{corollary}
	\begin{proof}
	Applying Corollary~\ref{cor1} to $f$ and the constant function of value $f(\tau(x_0))$, we get
	$$Tf(x_0)=Tg(x_0)=T0(x_0)+{\alpha}^{-1}f(\tau(x_0))=cf(\tau(x_0))+\phi(x_0).$$
	\end{proof}
\medskip

\subsection{Proof of Theorem~\ref{teo}}\label{ss-3.3}
Recalling the notation of the statement of Theorem~\ref{teo} we set $c:={\alpha}^{-1}=T1-T0$ and $\phi=T0$. Since $\|\phi|_S<1$, in particular $\phi(x_1)-\phi (x_2)<{d_{\mathcal{X}}(x_2,x_1)}$ for all $x_1,x_2\in \mathcal{X}$ such that $x_1\neq x_2$. It is easy to check that we can use $\phi$ to define a quasi-metric on $\mathcal{X}$ as in Proposition~\ref{ai}, obtaining that ${{d_{\mathcal{X}}'(x_1,x_2)=d_{\mathcal{X}}(x_1,x_2)+\phi(x_1)-\phi(x_2)}}$ is a quasi-metric on $\mathcal{X}$ such that $(\mathcal{X},d_{\mathcal{X}})$ is almost isometric to $(\mathcal{X},d_{\mathcal{X}}')$. In order to modify the isomorphism $T$, we define the following mappings:
\begin{itemize}
\item $ R:SC_{1^-}^1(\mathcal{X},d_{\mathcal{X}}) \to SC_{1^-}^1(\mathcal{X},d_{\mathcal{X}}')\,\,\text{by }\, R(g)= g-\phi$ ; \smallskip
\item $ S:SC_{1^-}^1(\mathcal{X},d_{\mathcal{X}}') \to SC_{1^-}^1(\mathcal{X},\alpha d_{\mathcal{X}}')\,\,\text{by}\,S(h)= \alpha h$ ; \smallskip and
\item $\hat{T}:SC_{1^-}^1(\mathcal{Y},d_{\mathcal{Y}}) \to SC_{1^-}^1(\mathcal{X},\alpha d_{\mathcal{X}}')\,\,\text{by}\,
\hat{T}(f)= S\circ R \circ T (f).$ \smallskip
\end{itemize}
Thanks to Proposition~\ref{isomorph} the mapping $R$ is well-defined: indeed, the same arguments used in Proposition~\ref{isomorph} are valid for the quasi-metric $d_{\mathcal{X}'}$ (which comes from a Finsler structure, thanks to Proposition~\ref{tausmooth}). We shall prove that both $\hat{T}$ and $\hat{T}^{-1}$ act as composition operators whenever their images are evaluated on the dense sets $\mathcal{X}'$ and $\mathcal{Y}'$ of Corollary~\ref{densos} respectively. Indeed, given $f\in SC_{1^-}^1(\mathcal{Y})$ and $x_0\in \mathcal{X}'$, we have:
		$$\hat{T}f(x_0)=S\circ R \circ T (f) (x_0)=\alpha \left({\alpha}^{-1} f(\tau(x_0))+\phi(x_0)-\phi(x_0)\right) =f(\tau(x_0)).\smallskip$$
On the other hand, for $g\in SC_{1^-}^1(\mathcal{X},\alpha d_{\mathcal{X}}')$ and $y_0\in \mathcal{Y}'$, we have $$\hat{T}^{-1} g(y_0)=T^{-1}\circ R^{-1} \circ S^{-1}(g)(y_0)=T^{-1}\left({\alpha}^{-1}g+\phi \right)(y_0).$$
Since ${\alpha}^{-1}g+\phi\in SC_{1^-}^1(\mathcal{X},d_{\mathcal{X}})$, there exists $f\in SC_{1^-}^1(\mathcal{Y},d_{\mathcal{Y}})$ such that $Tf={\alpha}^{-1}g+\phi.$ Then, denoting $x_0=\tau^{-1}(y_0)$ we obtain that ${\alpha}^{-1}g(x_0)+\phi(x_0)=Tf(x_0)={\alpha}^{-1}f(y_0)+\phi(x_0),$  whence $f(y_0)=g(x_0)$. Finally
		$$\hat{T}^{-1} g(y_0)=T^{-1}(Tf)(y_0)=g(x_0)=g(\tau^{-1}(y_0)).$$
		
Let us now prove that $\tau:(\mathcal{X}',\alpha d_{\mathcal{X}}')\to (\mathcal{Y}',d_{\mathcal{Y}})$ is an isometry. To this end, let $x_1,x_2 \in \mathcal{X}'$, $y_1=\tau(x_1)$ and $y_2=\tau(x_2)$. Take $\lambda\in (0,1)$ and $\varepsilon>0$ such that $\lambda+\varepsilon < 1$, and consider the function $f_\lambda(\cdot)=\lambda {d_{\mathcal{Y}}(y_1,\cdot)}$. Note that  $\|f_\lambda|_S=\lambda<1$, so we can apply Corollary~\ref{smoothslip} (smooth approximation of semi-Lipschitz functions), obtaining $g\in C^1(\mathcal{Y})$ such that $|g(y)-f_\lambda(y)|<\varepsilon$ for all $y\in \mathcal{Y}$ and $\|g|_S\leq \lambda+\varepsilon< 1$.
The second condition guarantees that $g\in SC_{1^-}^1(\mathcal{Y},d_{\mathcal{Y}})$. From the first condition it follows that {$|g(y_1)|<\varepsilon$} and ${g(y_2)}>\lambda d_{\mathcal{Y}}(y_1,y_2) - \varepsilon$. We deduce:
$$\alpha d_{\mathcal{X}}'(x_1,x_2) \geq {\hat{T}g(x_2)-\hat{T}g(x_1) = g(y_2)-g(y_1)} \geq \lambda d_{\mathcal{Y}}(y_1,y_2)-2\varepsilon$$
for any $\varepsilon>0$ such that $\varepsilon+\lambda < 1$. Consequently, $	\alpha d_{\mathcal{X}}'(x_1,x_2)\geq \lambda d_{\mathcal{Y}}(y_1,y_2)$, for any $\lambda\in (0,1)$. Therefore $$\alpha d_{\mathcal{X}}'(x_1,x_2)\geq d_{\mathcal{Y}}(y_1,y_2).$$
A similar argument holds for the reverse inequality. Take $\lambda\in (0,1)$, $\varepsilon>0$ such that $\lambda+\varepsilon< 1$, and consider $f_\lambda(\cdot)=\lambda {d_{\mathcal{X}}(x_1,\cdot)}$. Applying again Corollary~\ref{smoothslip} we get $g\in C^1(\mathcal{X})$ such that $|g(x)-f_\lambda(x)|<\varepsilon$ for all $x\in \mathcal{X}$ and $\|g|_S\leq \lambda+\varepsilon< 1$.
	Consider $\tilde{g}=\alpha(g-\phi)+\alpha \lambda \phi {(x_1)}\in C^1(\mathcal{X})$. Moreover, $\tilde{g}\in SC_{1^-}^1(\mathcal{X},\alpha d_{\mathcal{X}}')$, since $\tilde{g}=S\circ R (g) +\alpha \lambda \phi {(x_1)}.$
Let us now note that
$$ |\tilde{g}{(x_2)}-\lambda \alpha d_{\mathcal{X}}'(x_1,x_2)| =|\alpha (g{(x_2)}-\lambda d_{\mathcal{X}}(x_1,x_2)) -\alpha \phi {(x_2)}-\lambda \alpha\phi{(x_2)}| \leq \alpha \varepsilon + \alpha (1-\lambda)|\phi{(x_2)}|,$$
which together with $ |\tilde{g}{(x_1)}|=|\alpha g{(x_1)}-\alpha \phi {(x_1)}+\alpha \lambda \phi{(x_1)}|\leq \alpha \varepsilon +\alpha (1-\lambda)|\phi{(x_1)}|,$ yields
		\begin{align*}
		    d_{\mathcal{Y}}(y_1,y_2)&\geq \hat{T}^{-1}\tilde{g}{(y_2)}-\hat{T}^{-1}\tilde{g}{(y_1)}=\tilde{g}{(x_2)}-\tilde{g}{(x_1)}\\
		    &\geq \lambda \alpha d_{\mathcal{X}}'(x_1,x_2)-2\alpha \varepsilon -\alpha (1-\lambda)(|\phi(x_1)|+|\phi(x_2)|),
		\end{align*}
		for any $\varepsilon>0$ such that $\varepsilon+\lambda < 1$.
	Hence, $$d_{\mathcal{Y}}(y_1,y_2)\geq \alpha d_{\mathcal{X}}'(x_1,x_2)-\alpha (1-\lambda)(|\phi(x_1)|+|\phi(x_2)|),$$
	for any $\lambda\in (0,1)$, and therefore $d_{\mathcal{Y}}(y_1,y_2)\geq \alpha d_{\mathcal{X}}'(x_1,x_2)$. \smallskip

	We conclude that $\tau:(\mathcal{X}',\alpha d_{\mathcal{X}}')\to (\mathcal{Y}',d_{\mathcal{Y}})$ is an isometry. It is easy to check that $(\mathcal{X},\alpha d_{\mathcal{X}}')$ is also bicomplete, as $(d_{\mathcal{X}}')^s\leq 2d_{\mathcal{X}}^s$. Then, the isometry $\tau$ between the symmetrizations of $(\mathcal{X}',\alpha d_{\mathcal{X}}')$ and $(\mathcal{Y}',d_{\mathcal{Y}})$ extends to an isometry between $(\mathcal{X},\alpha d_{\mathcal{X}}')$ and $(\mathcal{Y},d_{\mathcal{Y}})$. By continuity, we obtain that for any $f\in SC_{1^-}^1(\mathcal{Y})$ and $x\in \mathcal{X}$,	
	$$Tf(x)=c\cdot f(\tau(x)) +\phi(x).$$
	
	Moreover, since $\tau$ is an almost isometry between the Finsler manifolds $(\mathcal{X},d_{\mathcal{X}})$ and $(\mathcal{Y},{\alpha}^{-1}{d_{\mathcal{Y}}})$, both $\tau$ and $\phi$ are smooth, thanks to Proposition~\ref{tausmooth}.  \hfill\qed
	\bigskip
	\subsection{Functional characterization of isometries and almost isometries}
Let us recall from \cite{CJ} the following definition:
\begin{definition}[almost unital isomorphism]\label{almostunital}
An isomorphism of convex partially ordered sets $${T:SC_{1^-}^1(\mathcal{Y})\to SC_{1^-}^1(\mathcal{X})}$$ is called \textit{almost unital} if ${T1-T0=1}$.
\end{definition}

Applying the results of the previous section we obtain:

\begin{corollary}[Characterization of {strict} Finsler almost isometries]\label{strictalmostiso}
Let $(\mathcal{X},d_{\mathcal{X}})$ and $(\mathcal{Y},d_{\mathcal{Y}})$ be connected, second countable Finsler manifolds, which are bicomplete (with their respective Finsler distances). Then, there is a {strict almost isometry} between $(\mathcal{X},d_{\mathcal{X}})$ and $(\mathcal{Y},d_{\mathcal{Y}})$ if and only if there exists an almost unital isomorphism $${T:SC_{1^-}^1(\mathcal{Y})\to SC_{1^-}^1(\mathcal{X})}.$$
In particular, for any such isomorphism, there exist a diffeomorphism $\tau:\mathcal{X}\to \mathcal{Y}$ and a smooth function ${\phi\in SC_{1^-}^1(\mathcal{X})}$ such that $Tf= f\circ \tau+\phi$ for all $f\in SC_{1^-}^1(\mathcal{Y})$.
	\end{corollary}
\begin{proof}
	The ``if" implication follows directly from Theorem~\ref{teo} and Definition~\ref{almostunital}, and the ``only if" part is Proposition~\ref{isomorph}.
\end{proof}
	\medskip
{Using Proposition \ref{compact}, we obtain the following characterization of almost isometries between compact Finsler manifolds:
\begin{corollary}[Characterization of almost isometries between compact Finsler manifolds]\label{compactalmost}
	Let $(\mathcal{X},d_{\mathcal{X}})$ and $(\mathcal{Y},d_{\mathcal{Y}})$ be compact, connected, second countable Finsler manifolds, which are bicomplete (with their respective Finsler distances). Then  $(\mathcal{X},d_{\mathcal{X}})$ and $(\mathcal{Y},d_{\mathcal{Y}})$ are almost isometric if and only if there exists an almost unital isomorphism $${T:SC_{1^-}^1(\mathcal{Y})\to SC_{1^-}^1(\mathcal{X})}.$$
	In particular, for any such isomorphism, there exist a diffeomorphism $\tau:\mathcal{X}\to \mathcal{Y}$ and a smooth function ${\phi\in SC_{1^-}^1(\mathcal{X})}$ such that $Tf= f\circ \tau+\phi$ for all $f\in SC_{1^-}^1(\mathcal{Y})$.
\end{corollary}	
}

	If we focus on isometries, we obtain:
	\begin{corollary}[Characterization of Finsler isometries]\label{isomGeneral}
Let $(\mathcal{X},d_{\mathcal{X}})$ and $(\mathcal{Y},d_{\mathcal{Y}})$ be connected, second countable Finsler manifolds which are  bicomplete (with their respective Finsler distances). Then, $(\mathcal{X},d_{\mathcal{X}})$ and $(\mathcal{Y},d_{\mathcal{Y}})$ are \emph{isometric} if and only if there exists an isomorphism ${T:SC_{1^-}^1(\mathcal{Y})\to SC_{1^-}^1(\mathcal{X})}$ such that
$${\|Tf|_S=\|f|_S}\,,\quad \text{for all } \,\,f\in SC_{1^-}^1(\mathcal{Y}).$$
Moreover, for any such isomorphism, there exist a diffeomorphism $\tau:\mathcal{X}\to \mathcal{Y}$ and ${\beta \in \mathbb{R}}$ such that ${Tf= f\circ \tau+\beta}$ for all $f\in SC_{1^-}^1(\mathcal{Y})$.
\end{corollary}

\begin{proof}
	If $\tau$ is an isometry between $(\mathcal{X},d_{\mathcal{X}})$ and $(\mathcal{Y},d_{\mathcal{Y}})$, then $f\mapsto T(f):=f\circ \tau$ is an isomorphism between convex, partially ordered structures that satisfies ${\|Tf|_S=\|f|_S}$ for all $f\in SC_{1^-}^1(\mathcal{Y})$. \smallskip

Conversely, we can apply Theorem~\ref{teo} to the isomorphism $T$, and since $\|T0|_S=0$, the function $T0$ is a constant, so the quasi-metric $d_{\mathcal{X}}'$ induced by $T0$ is the same as $d_{\mathcal{X}}$. In addition, $\alpha$ must be $1$ for $T$ to preserve semi-Lipschitz constants, and therefore $\tau:(X,d_{\mathcal{X}})\to (Y,d_{\mathcal{Y}})$ is an isometry.
\end{proof}
\medskip
In the particular case of reversible Finsler manifolds, Theorem~\ref{teo} can be restated as follows.
\begin{corollary}\label{beta}
Let $(\mathcal{X},d_{\mathcal{X}})$ and $(\mathcal{Y},d_{\mathcal{Y}})$ be connected, second countable, reversible complete Finsler manifolds and ${T:C_{1^-}^1(\mathcal{Y})\to C_{1^-}^1(\mathcal{X})}$ be an isomorphism of convex partially ordered sets. Then, there exist $\alpha >0$, $\beta \in \mathbb{R}$ and a bijection $\tau:\mathcal{X}\to \mathcal{Y}$ such that:
		\begin{itemize}
   	\item[(i)] $(\mathcal{Y},d_{\mathcal{Y}})$ and $(\mathcal{X},\alpha d_{\mathcal{X}})$ are isometrically diffeomorphic via $\tau$.\smallskip
	\item[(ii)] For every $f\in C_{1^-}^1 (Y)$ we have $Tf=c\cdot (f\circ \tau) +\beta$, where $c={\alpha}^{-1}$ and $\beta=T0$.
\end{itemize}
		\end{corollary}
		\begin{proof}
		It follows from Theorem~\ref{teo}. (Since all involved distances are symmetric, $\phi$ must be constant.)
		\end{proof}
\smallskip
Therefore we obtain the following characterization of isometries for reversible Finsler Manifolds.	
\begin{corollary}[Characterization of isometries for reversible Finsler manifolds]\label{isomRev}
		Let $(\mathcal{X},d_{\mathcal{X}})$ and $(\mathcal{Y},d_{\mathcal{Y}})$ connected, second countable, reversible complete Finsler manifolds. Then the manifolds $(\mathcal{X},d_{\mathcal{X}})$, $(\mathcal{Y},d_{\mathcal{Y}})$ are isometric if and only if there exists an almost unital isomorphism ${T:C_{1^-}^1(\mathcal{Y})\to C_{1^-}^1(\mathcal{X})}$. Moreover, for any such isomorphism there exist a diffeomorphism ${\tau:\mathcal{X}\to \mathcal{Y}}$ and $\beta \in \mathbb{R}$ such that $$Tf= f\circ \tau +\beta\quad \text{ for all } \, f\in C_{1^-}^1(\mathcal{Y}).$$	
		\end{corollary}
\smallskip
Note that the isomorphism of partially ordered sets in the above Corollary preserves Lipschitz constants, and can be replaced by $\tilde{T}=T-\beta$ in order to extend linearly to the spaces of $C^1$-smooth Lipschitz functions, denoted by $C_{\mathrm{Lip}}^1(\mathcal{Y})$ and $C_{\mathrm{Lip}}^1(\mathcal{X})$ respectively. Therefore, for the particular case of reversible Finsler manifolds we can reformulate Corollary~\ref{isomRev} as follows:
\begin{corollary}
Let $(\mathcal{X},d_{\mathcal{X}})$ and $(\mathcal{Y},d_{\mathcal{Y}})$ connected, second countable, reversible complete Finsler manifolds. Then $(\mathcal{X},d_{\mathcal{X}})$ and $(\mathcal{Y},d_{\mathcal{Y}})$ are isometric if and only if there exists a linear, order and semi-norm preserving bijection $T:(C_{\mathrm{Lip}}^1(\mathcal{Y}),\|\cdot\|_\mathrm{Lip}) \to (C_{\mathrm{Lip}}^1(\mathcal{X}),\|\cdot\|_\mathrm{Lip})$. Moreover, for any such bijection there exist a diffeomorphism ${\tau:\mathcal{X}\to \mathcal{Y}}$ such that $Tf= f\circ \tau$ for all $ f\in C_{\mathrm{Lip}}^1(\mathcal{Y})$.
\end{corollary}
		
	
\bigskip
\subsection*{Acknowledgement}
Part of this work has been done during a research stay of A. Daniildis (as Gaspard Monge invited professor) and F. Venegas (as Research trainee) at INRIA (\'Equipe Tropicale) and CMAP of \'Ecole Polytechnique (France).  A preliminary version of this work has been presented by the third author at the conference ``Function Theory on Infinite Dimensional Spaces XV", held at the Complutense University of Madrid (February 2018). This author wishes to thank the organizing committee of the aforementioned event for hospitality.

%
%
%
%
%
%

%
%



\begin{thebibliography}{9}
	\bibitem{BCS}
	D. Bao, S-S Chern, and Z. Shen.
	\newblock {\em An introduction to Riemann-Finsler geometry}, volume 200 (2012).
	\newblock Springer Science \& Business Media.
	
		\bibitem{CC}
	F. Cabello-S{\'a}nchez and J. Cabello-S{\'a}nchez.
	\newblock \emph{Some preserver problems on algebras of smooth functions.}
	\newblock Ark. Mat. \textbf{48(2)} (2010), 289--300.
	
	
	\bibitem{CC11}
	F. Cabello-S{\'a}nchez and J. Cabello-S{\'a}nchez.
	\newblock \emph{Nonlinear isomorphisms of lattices of {L}ipschitz functions.}
	\newblock Houston J. Math \textbf{37(1)} (2011), 181--202.
	
	
	\bibitem{CJ} J. Cabello-S\'anchez and J. A. Jaramillo,
	\emph{A functional representation of almost isometries},
	J. Math. Anal. Appl. \textbf{445} (2017), 1243--1257.
	
	\bibitem{C}
	S. Cobzas.
	\newblock {\em Functional analysis in asymmetric normed spaces} (2012).
	\newblock Springer Science \& Business Media.
	
	\bibitem{DH} S. Deng and Z. Hou, Zixin
	\emph{The group of isometries of a Finsler space},
	Pacific J. Math. \textbf{207} (2002), 149--155.
	
	\bibitem{GJ} M. I. Garrido and J. A. Jaramillo,
	\emph{Variations on the Banach-Stone theorem},
	Extracta Math. \textbf{17} (2002), 351--383.
	
	\bibitem{GJR-10} M. I. Garrido, J. A. Jaramillo and Y. Rangel,
	\emph{Lip-density and algebras of Lipschitz functions on metric spaces},
	Extracta Math. \textbf{25} (2010), 249--261.
	
		
	\bibitem{GJR-13}
	M. I. Garrido, J. A. Jaramillo and Y. Rangel,
	\newblock \emph{Smooth approximation of {L}ipschitz functions on {F}insler manifolds.}
	\newblock J. Func. Spaces Appl. \textbf{2013} (2013), 10p. 		
	
	
	
	\bibitem{JLP} M.A. Javaloyes, L. Lichtenfelz and P. Piccione,
	\emph{Almost isometries of non-reversible metrics with applications to stationary spacetimes},
	J. Geom. Phys. \textbf{89} (2015) 38--49.
	
	\bibitem{M} S. B. Myers,
	\emph{Algebras of differentiable functions},
	Proc. Amer. Math. Soc. \textbf{5} (1954), 917--922.
	
	\bibitem{MS} S. B. Myers and N. E. Steenrod,
	\emph{The group of isometries of a Riemannian manifold},
	Ann. of Math. \textbf{(2) 40} (1939), 400--416.
	
		
	\bibitem{N} M. Nakai,
	\emph{Algebras of some differentiable functions on Riemannian manifolds},
	Japan. J. Math. \textbf{29} (1959),  60--67.

	



	
\end{thebibliography}
\end{document}